\documentclass[12pt,reqno]{amsart}
\usepackage{amsmath}
\usepackage{amssymb}
\usepackage{amstext}
\usepackage{mathrsfs}
\usepackage{a4wide}
\usepackage{graphicx}
\usepackage{bbm}
\usepackage{verbatim}
\usepackage{bm}
\usepackage{graphicx}
\usepackage{tikz}
\usepackage{pgfplots}
\usepackage{titletoc}

\allowdisplaybreaks \numberwithin{equation}{section}
\usepackage{color}
\usepackage{cases}
\usepackage{hyperref}
\hypersetup{hypertex=true,
	colorlinks=true,
	linkcolor=black,
	anchorcolor=black,
	citecolor=black}

\numberwithin{equation}{section}

\newtheorem{theorem}{Theorem}[section]
\newtheorem{proposition}[theorem]{Proposition}
\newtheorem{corollary}[theorem]{Corollary}
\newtheorem{lemma}[theorem]{Lemma}

\theoremstyle{definition}

\newtheorem{definition}[theorem]{Definition}

\theoremstyle{remark}
\newtheorem{remark}[theorem]{Remark}

\newcommand{\ep}{\varepsilon}

\begin{document}
		
	\title[Traveling-wave solutions for the gSQG equation]{Slow traveling-wave solutions for the generalized surface quasi-geostrophic equation }
	
	\author{Daomin Cao, Shanfa Lai, Guolin Qin}
	
	\address{Institute of Applied Mathematics, Chinese Academy of Sciences, Beijing 100190, and University of Chinese Academy of Sciences, Beijing 100049,  P.R. China}
	\email{dmcao@amt.ac.cn}
	
	\address{Institute of Applied Mathematics, Chinese Academy of Sciences, Beijing 100190, and University of Chinese Academy of Sciences, Beijing 100049,  P.R. China}
	\email{laishanfa@amss.ac.cn}
	
	\address{Institute of Applied Mathematics, Chinese Academy of Sciences, Beijing 100190, and University of Chinese Academy of Sciences, Beijing 100049,  P.R. China}
	\email{qinguolin18@mails.ucas.ac.cn}

	\begin{abstract}
		In this paper, we systematically study the existence, asymptotic behaviors, uniqueness, and nonlinear orbital stability of traveling-wave solutions with small propagation speeds for the generalized surface quasi-geostrophic (gSQG) equation.  Firstly we obtain the existence of a new family of global solutions via the variational method. Secondly we show the uniqueness of maximizers under our variational setting. Thirdly by using the variational framework, the uniqueness of maximizers and a concentration-compactness principle  we establish some stability theorems. Moreover, after a suitable transformation, these solutions constitute the desingularization of  traveling point vortex pairs.
	\end{abstract}
	
	\maketitle{\small{\bf Keywords:} The gSQG equation; Traveling-wave solutions; Variational methods; Existence and uniqueness of maximizer; Nonlinear stability.   \\
		
		{\bf 2020 MSC} Primary: 76B47; Secondary: 76B03, 35A02.}
	\setcounter{tocdepth}{3}
	\\
	\\
	
	\section{Introduction and Main results}\label{sec1}

	In this paper, we are concerned with the  following generalized surface quasi-geostrophic (gSQG) equation
	\begin{align}\label{1-1}
		\begin{cases}
			\partial_t\theta+\mathbf{u}\cdot \nabla \theta =0&\text{in}\ \mathbb{R}^2\times (0,T),\\
			\ \mathbf{u}=\nabla^\perp(-\Delta)^{-s}\theta     &\text{in}\ \mathbb{R}^2\times (0,T),\\
		\end{cases}
	\end{align}
	where $0< s< 1$, $\theta(x,t):\mathbb{R}^2\times (0,T)\to \mathbb{R}$ is the active scalar being transported by the velocity field $\mathbf{u}(x,t):\mathbb{R}^2\times (0,T)\to \mathbb{R}^2$ generated by $\theta$, and $(a_1,a_2)^\perp=(a_2,-a_1)$. The operator $(-\Delta)^{-s}$ is defined by
	\begin{equation*}
		(-\Delta)^{-s}\omega(x)=\mathcal{G}_s\omega(x)=\int_{\mathbb{R}^2}G_s(x-y)\omega(y)dy,
	\end{equation*}
	where $G_s$ is the fundamental solution of $(-\Delta)^{s}$ in $\mathbb{R}^2$ given by
	\begin{equation*}
		G_s(z)=
			\frac{c_s}{|z|^{2-2s}}, \ \ \ c_s=\frac{\Gamma(1-s)}{2^{2s}\pi\Gamma(s)}.	
	\end{equation*}

	When $s={1}/{2}$,  \eqref{1-1} corresponds to the inviscid surface quasi-geostrophic (SQG) equation, which models the evolution of the temperature from a general quasi-geostrophic system for atmospheric and atmospheric flows (see e.g. \cite{Con, Lap1}).  The  SQG equation has received  extensive concern as a simplified model for the three-dimensional Euler equations since \cite{Con}. Formally at least, in the limit $s\uparrow1$ we obtain the well-known two-dimensional Euler equation in vorticity formulation \cite{MB}. For general $0<s<1$, \eqref{1-1} was proposed by C\'{o}rdoba et al. in \cite{Cor} as an interpolation between the Euler equation and the SQG equation.
	
	The global well-posedness for the Cauchy problem for two-dimensional incompressible Euler equation (i.e., $s=1$ in \eqref{1-1}) has been well studied. Global well-posedness  for Cauchy problems with initial data in $L^1\cap L^\infty$ was established by Yudovich \cite{Yud}. The $L^1$ assumption can be replaced by an appropriate symmetry condition thanks to the work of Elgindi and Jeong \cite{Elg}.  We refer to \cite{Elg, MB} and references therein for more discussions. However, to the best of our knowledge, the problem of whether the gSQG system presents finite time singularities or there is global well-posedness of classical solutions is still open; see \cite{Cas2, Chae, Hes, Kis1, Kis2} and references therein for more details.
	
	For  vortex patch type global solutions, the first non-trivial example was constructed in \cite{Has} for $\frac{1}{2}<s<1$ by using the contour dynamics equation and bifurcation theory.  Numerous results on the vortex patch type solutions for the gSQG equations were then obtained in different situations (see e.g. \cite{Cas1, Cas4, Cora, de1, Gar2, Go1, Has2, HM}). In \cite{Cas2}, Castro et al. established the first  result of existence on global smooth solutions for the gSQG equation by developing a bifurcation argument from a specific radially symmetric function. In \cite{Gra}, Gravejat and Smets, for the first time, proved the existence of smooth translating vortex pairs for the SQG equation. This result was then  generalized to the gSQG equation with $s\in(0,1)$ by Godard-Cadillac \cite{Go0}. In \cite{Asqg}, Ao et al. successfully  constructed traveling and rotating smooth solutions to the gSQG equation with $s\in(0,1)$ by the Lyapunov-Schmidt reduction method.
	
	In this paper, we are interested in traveling-wave solutions for the gSQG equation. Up to a rotation, we may assume, without loss of generality, that these waves have a negative speed $-\mathcal W$ in the vertical direction, so that
	\begin{equation*}
		\theta(x,t)=\omega(x_1, x_2+\mathcal Wt).
	\end{equation*}
	In this setting, the first equation in \eqref{1-1} is also reduced to a stationary equation
	\begin{equation}\label{1-2}
		\nabla^\perp(\mathcal{G}_s\omega-\mathcal Wx_1)\cdot\nabla\omega=0,
	\end{equation}
	which has a weak form
	\begin{equation}\label{1-3}
		\int_{\mathbb{R}^2}\omega\nabla^\perp(\mathcal{G}_s\omega-\mathcal Wx_1)\cdot\nabla \varphi dx=0, \ \ \ \forall\,\varphi\in C_0^\infty(\mathbb{R}^2).
	\end{equation}

     In the study of traveling-wave solutions for ideal incompressible fluids, translating vortex pairs is  the main concern. The literature on vortex pairs can be traced back to the work of Pocklington \cite{Poc} in 1895. In 1906, Lamb \cite{Lamb} founded an explicit solution for the Euler equation which is now generally referred to as the Lamb dipole or Chaplygin-Lamb dipole; see also \cite{Mel}.   Besides those exact solutions, the existence (and abundance) of translating vortex pairs for the Euler equation has been rigorously established in \cite{Amb, Bu0, SV, T1, Yang} and so on. As mentioned above, for the gSQG equation, some examples of traveling-wave solutions were constructed in \cite{Asqg, CQZZ1, CQZZ, Go0, Gra,  Has2, HM}.

   In this paper, we will  obtain a new family of traveling-wave solutions for the gSQG  equation and further investigate their  asymptotic behaviors, uniqueness, and nonlinear orbital stability.

   \subsection{Main results}

    As pointed out by Arnol'd \cite{Ar1}, a natural way of obtaining solutions to the stationary problem \eqref{1-2} is to impose that $\omega$ and $\mathcal{G}_s\omega-\mathcal Wx_1$ are (locally) functional dependent. That is, one may impose that
  $$\omega=f(\mathcal{G}_s\omega-\mathcal Wx_1),$$ for some Borel measurable function $f:\mathbb{R}\to \mathbb{R}$.

  Usually $f$ is supposed to satisfy the following hypotheses
  \begin{itemize}
  	\item[$(\mathbf{H_{1}}).$] $f(0)=0$, $f$ is   nonnegative and strictly increasing for $t>0$;
  	\item[$(\mathbf{H_{2}}).$] $\lim_{t\to 0_+} t^{-\frac{1}{1-s}}f(t)=+\infty$ and  $\lim_{t\to +\infty} t^{-\frac{1}{1-s}}f(t)=0$.
  \end{itemize}

  Our first main result concerns the existence of traveling-wave solutions with slow traveling speeds and the fine asymptotic behaviors of these solutions. For convenience, we will take $\mathcal W=W\ep^{3-2s}$ for some constant $W>0$ and some small $\ep>0$. Let $\mathbf{e}_2=(0,1)$ be the unit vector along the $x_2$-direction. Let $\mathbb{R}_+^2:=\{x\in\mathbb{R}^2\mid x_1>0\}$   be the right half plane and  $1_S$ represents  the characteristic function of a set $S$. Denote by $\text{spt}(\omega)$  the support of a function $\omega$.   For fixed $W>0,\kappa>0$ denote
 \begin{equation}\label{d0}
 	d_0=\left(\frac{(1-s)c_s\kappa}{2^{2-2s} W}\right)^{\frac{1}{3-2s}}.
 \end{equation}

  Our first result is as follows.
  \begin{theorem}\label{thmE}
  	  Let $0<s<1$, $W>0$ $\kappa>0$ be given. Suppose that $f$ is a measurable function satisfying $(\mathbf{H_{1}}) $ and $(\mathbf{H_{2}}) $ and $f\in C^{1-2s }$ if $0<s<\frac{1}{2}$. Then there is a number $\ep_0>0$ small such that for any $\ep\in (0, \ep_0)$, \eqref{1-1} has a traveling-wave solution of the form $\theta_\ep(x,t)=\omega_{tr,\ep}(x+W\ep^{3-2s} t\mathbf{e}_2)$ for some function $\omega_{tr,\ep}\in L^\infty(\mathbb R^2)$ in the sense that $\omega_{tr,\ep}$ solves \eqref{1-3} with $\mathcal W=W\ep^{3-2s}$. Moreover, $\omega_{tr,\ep}$ has the following properties:
  	  \begin{itemize}
  	  	\item[(i)] $\omega_{tr,\ep}$ is odd in $x_1$ and even in $x_2$. That is,
  	  	$$\omega_{tr,\ep}(-x_1,x_2)=-\omega_{tr,\ep}(x_1,x_2),\ \ \omega_{tr,\ep}(x_1,-x_2)= \omega_{tr,\ep}(x_1,x_2),\ \ \forall\ x\in\mathbb{R}^2;$$
  	  	\item[(ii)] It holds for some constant $\mu_\ep$
  	  	$$\omega_{tr,\ep}=f(\mathcal G_s\omega_{tr,\ep}-W\ep^{3-2s}x_1-\mu_\ep),\ \ \text{in}\ \mathbb{R}_+^2; $$
  	  	\item[(iii)] Let $\omega_\ep:=\omega_{tr,\ep} 1_{\mathbb{R}_+^2}$ and denote the center of mass of $\omega_\ep$ by $x_\ep:=\kappa^{-1}\int x\omega_\ep$. Then, there hold $$\int_{\mathbb{R}_+^2} \omega_\ep=\kappa  \ \text{and}\ \ep x_\ep=(d_0,0)+o(1), \ \ \text{as}\ \ep\to0,$$
  	  	where  $d_0$ s given by \eqref{d0}.
  	  	Furthermore, there is a constant $R>0$ independent of $\ep$ such that $\text{spt}(\omega_\ep)$ is contained in the  disk with center $x_\ep$ and radius $R$.
  	  \end{itemize}
  \end{theorem}
  \begin{remark}
  	The assumption $f\in  C^{1-2s}$ in the case  $0<s<\frac{1}{2}$ is used to improve the regularity of $\mathcal G_s\omega_\ep$ (see e.g. Propositions 2.8 and 2.9 in \cite{Sil}) so that the integral in \eqref{1-3} makes sense. 
  \end{remark}
   \begin{remark}
   	Typical examples of $f$ satisfying the assumptions in Theorem \ref{thmE} includes any $C^{0,1}$ smooth bounded strictly increasing functions with $f(0)=0$, such as $f(t)=\arctan(t_+)$ as well as some unbounded functions, for example,  $f(t)=t_+^p$ with $p\in(1,\frac{1}{1-s})$. Here $t_+$ means $\max\{0, t\}$. We will obtain finer asymptotic behaviors  and prove the uniqueness and stability in the later case.
   \end{remark}
\begin{remark}
	As we shall see in the proof of Lemma \ref{lem2-21}, up to some translation, $\omega_\ep$ tends to a nontrivial function $\omega_0$ in $L^1\cap L^{2-s}(\mathbb{R}^2)$, which is a maximizer of the limiting problem considered in subsection \ref{sec2-1}. Therefore, the amplitude of the solutions obtained in Theorem \ref{thmE} does not vanish as $\ep\to0$.
\end{remark}

  Let $G_s^+(x,y):=\frac{c_s}{|x-y|^{2-2s}}-\frac{c_s}{|x-\bar y|^{2-2s}}$ with $\bar y=(-y_1,y_2)$ and define $$\mathcal G_s^+\omega:=\int_{\mathbb{R}_+^2} G_s^+(x,y) \omega(y) dy.$$ The proof of Theorem \ref{thmE} is based on a constrained maximization method. More precisely, take $J$ be defined by $J(t)=\int_0^t f^{-1}(\tau) d\tau$ and let  $B(x,r)$ stand for the disk with center $x$ and radius $r$ and $\kappa>0$ be a constant. We are to consider the problem of maximizing the following functional
  \begin{equation}\label{energy}
  	E_\ep(\omega):=\frac{1}{2}\int_{\mathbb{R}^2_+}\int_{\mathbb{R}^2_+}\omega(x)G_s^+(x,y)\omega(y)dxdy-W\ep^{3-2s}\int_{\mathbb{R}^2_+} x_1 \omega(x) dx  -\int_{\mathbb{R}^2_+} J(\omega(x)) dx,
  \end{equation}
 over the constraint
 \begin{equation}\label{constraint}
 \mathcal{A}_\ep:=\bigg\{\omega\in L^1\cap L^{2-s}(\mathbb{R}^2_+)\mid \omega\geq 0,\, \text{spt}(\omega)\subset B\left((\frac{d_0}{\ep},0), \frac{d_0}{2\ep}\right), \,  \int_{\mathbb{R}^2_+} \omega(x) dx=\kappa\bigg\}.
 \end{equation}

 	Consider the following
 maximization problem:
 \begin{equation}\label{max**}
 	e_\ep:=\sup_{\omega\in \mathcal{A}_\ep} E_\ep(\omega).
 \end{equation}
 For the above maximizing problem we have the following result.
  \begin{theorem}\label{thmax*}
  	Let $0<s<1$ and $W>0$. Suppose that $f$ is a measurable function satisfying $(\mathbf{H_{1}}) $ and $(\mathbf{H_{2}}) $. Then there is a number $\ep_0>0$ small such that $e_\ep $ can be achieved for any $\ep\in (0, \ep_0)$,  that is $E_\ep$ admits a maximizer $\omega_\ep$ in $\mathcal{A}_\ep$.  Moreover, $\omega_\ep$ has the following properties:
  \begin{itemize}
  	\item[(i)] $\omega_{\ep }$ is Steiner symmetric with respect to some plane $\{x_2=const.\}$;
  	\item[(ii)] There is a constant $\mu_\ep$ such that
  	\begin{equation*} 
  		\omega_{\ep }=f\left( \mathcal{G}_s^+\omega_{\ep   }- W\ep^{3-2s}x_1-\mu_{\ep } \right), \ \ \ \text{in}\ B(\ep^{-1}(d_0,0), \ep^{-1} d_0/2);
  	\end{equation*}
  	\item[(iii)] The energy satisfies $$I_0+O(\ep^{2-2s})\leq E_\ep(\omega_\ep)\leq I_0,$$
  	where $I_0$ is given by \eqref{I0};
  	\item[(iv)] There exists a constant $0<C<+\infty$ independent of $\ep$ such that
  	\begin{equation*} 
  		\limsup_{\ep\to0_+}\|\omega_{\ep }\|_{L^\infty}\leq C.
  	\end{equation*}
  	\item[(v)] Denote the center of mass of $\omega_\ep$ by $x_\ep:=\kappa^{-1}\int x\omega_\ep$. Then, $$\ep x_\ep=(d_0,0)+o(1), \ \ \text{as}\ \ep\to0,$$
  where  $d_0$ is given by \eqref{d0}.
  Furthermore, there is a constant $R>0$ independent of $\ep$ such that $\text{spt}(\omega_\ep)$ is contained in the  disk with center $x_\ep$ and radius $R$.
  \end{itemize}
  \end{theorem}


  As we will see that Theorem \ref{thmE} can be derived from Theorem \ref{thmax*}. Indeed, we will prove in Lemma \ref{lem2-24} that if we further assume $f\in C^{1-2s }$ in the case $0<s<\frac{1}{2}$, then after a translation in $x_2$, $\omega_\ep(x)-\omega_\ep(\bar x)$ is the desired function $\omega_{tr,\ep}$ in Theorem \ref{thmE}.

  Next, we shall investigate the problem of uniqueness, which is crucial in the study of stability. We focus our attention on the case  $f=t_+^p$ for some $p\in(1,\frac{1}{1-s})$.  It can be seen that such $f$ satisfies all the assumptions in Theorem \ref{thmE} for $s\in(0,1)$. Therefore, for every $\ep>0$ small, Theorem \ref{thmE} ensures a traveling-wave solution with $\omega_{tr,\ep}(x)=\omega_\ep(x)-\omega_\ep(\bar x)$ and $\omega_\ep$ being a maximizer of $E_\ep$ over constraint $\mathcal{A}_\ep$. Inspired by the work on rotating stars \cite{Jan}, we will prove that $\omega_\ep$ is the unique maximizer in the sense that any maximizer of $E_\ep$ is a translation of $\omega_\ep$.

  For fixed $\ep_0$ as in Theorem \ref{thmax*} we denote by $\Sigma_\ep$ the set of all maximizers of $E_\ep$  over $\mathcal{A}_\ep$ for $\ep\in (0,\ep_0)$. Our second main result is as follows.
  \begin{theorem}\label{thmU}
  	Suppose that $f(t)=t_+^p$ for some $p\in(1,\frac{1}{1-s})$. Let $\ep_0$ be as in Theorem \ref{thmax*} and $\ep\in (0,\ep_0)$.  Let $\omega_\ep\in \Sigma_\ep$ be a maximizer as obtained in Theorem \ref{thmax*}. Then there is a number $\ep_1\in (0,\ep_0]$ such that for all $\ep\in (0,\ep_1)$, $$\Sigma_\ep=\{\omega_\ep(\cdot+c\mathbf{e}_2)\mid c\in\mathbb{R}\}.$$
  \end{theorem}

  For relative equilibria of fluids, there are fewer mathematical results available on the uniqueness. The first result was due to Amick and Fraenkel \cite{AF}, who proved that Hill's spherical vortex is the unique solution when viewed in a natural weak formulation by the method of moving planes. Later Amick and Fraenkel \cite{AF88} also established local uniqueness for Norbury's nearly spherical vortex. The uniqueness of the Chaplygin-Lamb dipole was shown by Burton \cite{B96} using a similar method as \cite{AF}. Recently, Jang and Seok \cite{Jan} proved the uniqueness of maximizers of a variational problem related to rotating binary stars, which inspired our proof of Theorem \ref{thmU}. As for the gSQG equation, to the best of our knowledge, the only result  on this issue is the recent work \cite{CQZZ1}, where a very special case was considered in order to apply the method of moving planes. Our result Theorem \ref{thmU} provides uniqueness in a wide range of cases.

  Our last main result concerns the orbital stability of the traveling-wave solutions obtained in Theorem \ref{thmE}.  We first prove a general stability theorem in a similar spirit as \cite{BNL13}, where the stability of vortex pairs for the 2D Euler equation was considered. To be precise, we will consider the maximization problem of the functional $$\tilde E_{\mathcal W}(\zeta):=\frac{1}{2}\int_{\mathbb{R}^2_+} \zeta(x)\mathcal{G}_s^+\zeta(x )  dx-\mathcal W \int_{\mathbb{R}^2_+}  x_1\zeta(x) dx$$ over the set $\overline{\mathcal{R}(\zeta_0)^w}$, which means the weak closure of the  rearrangement class of a given function $\zeta_0$. Using the concentrate compactness principle due to Lions \cite{L84}, we establish the compactness of maximizing sequence and derive a general  nonlinear stability theorem on the set of maximizers (see Theorem \ref{Sset}).

  In the case  $f=t_+^p$ for  $p\in(1,\frac{1}{1-s})$, Theorem \ref{thmE} provides a traveling-wave solution  with $\omega_{tr,\ep}(x)=\omega_\ep(x)-\omega_\ep(\bar x)$ and $\omega_\ep$ being a maximizer of $E_\ep$. To apply the stability theorem of the set of maximizers, we need to consider an auxiliary variational problem: maximize $\tilde E_{\mathcal W}$ with $\mathcal W=W\ep^{3-2s}$ over the set $\overline{\mathcal{R}(\omega_\ep)^w}$. By studying the asymptotic behaviors of maximizers and using the uniqueness result in Theorem \ref{thmU}, we are able to show that all the maximizers of the second variational problem are actually translations of $\omega_\ep$ in the $x_2$-direction. As a consequence, we obtain the orbital stability of $\omega_\ep$ by applying   the nonlinear stability theorem on the set of maximizers proved in Theorem \ref{Sset}.

  Roughly speaking, our stability result is as follows.
  \begin{theorem}\label{thmS}
  	Suppose that $f(t)=t_+^p$ for some $p\in(1,\frac{1}{1-s})$. Let $\omega_{tr,\ep}$ be the traveling-wave solution obtained in Theorem \ref{thmE}. Then for $\ep$ fixed small,  $\omega_{tr,\ep}$ is orbitally stable in the following sense:
  	for arbitrary  $M>0$ and $\eta>0$, there exists $\delta>0$ such that for non-negative function $\xi_0\in L^1\cap L^{\infty}(\mathbb{R}^2_+)$ with   $||\xi_0||_{\infty} <M$ and
  	\begin{equation}\label{1-4}
  		\begin{array}{ll}
  		\inf_{c\in \mathbb{R}}\Big{\{}\|\xi_0-\omega_{tr, \ep}(\cdot+c\mathbf{e}_2)\|_{L^1(\mathbb{R}^2_+)}+\|\xi_0-\omega_{tr,\ep}(\cdot+c\mathbf{e}_2)\|_{L^2(\mathbb{R}^2_+)}&\\
  			\qquad\qquad\qquad\qquad\qquad\qquad\qquad\quad\,\,+\|x_1(\xi_0-\omega_{tr,\ep}(\cdot+c\mathbf{e}_2))\|_{L^1(\mathbb{R}^2_+)}\Big{\}} \leq \delta,&
  			\end{array}
  	\end{equation}
  	if there exists a $L^{\infty}$-regular solution $\xi(t)$ with initial data $ \xi_0(x) $ for $t\in [0, T)$ with $0<T\leq \infty$, then all $t\in[0,T)$,
  	\begin{equation}\label{1-5}
  		\begin{array}{ll}
  		\inf_{c\in \mathbb{R}} \Big{\{} \|\xi(t)-\omega_{tr, \ep}(\cdot+c\mathbf{e}_2)\|_{L^1(\mathbb{R}^2_+)}+\|\xi(t)-\omega_{tr,\ep}(\cdot+c\mathbf{e}_2)\|_{L^2(\mathbb{R}^2_+)}&\\
  		\qquad\qquad\qquad\qquad\qquad\qquad\qquad\qquad\,+\|x_1(\xi(t)-\omega_{tr,\ep}(\cdot+c\mathbf{e}_2))\|_{L^1(\mathbb{R}^2_+)}\Big{\}}\leq \eta.&
  		\end{array}
  	\end{equation}
  \end{theorem}
  \begin{remark}
  	For the rigorous definition of $L^{\infty}$-regular solutions for the gSQG equation, please see Section 4. Once the uniqueness of other solutions  was established, one may apply the general stability theorem and the framework in this paper to obtain their orbital stability. Compared with the result in \cite{BNL13},  we admit perturbations with non-compact supports, which is achieved by   bringing in   the $L^1$-norm in our theorem.
  \end{remark}

  Much work has been done on the stability of steady solutions to the Euler equations, for which we refer the interested reader to \cite{Abe, Bu5, Bu6, CW,  Choi20, CL, CJ2, Gal, Wan} and references therein.

  \subsection{Desingularize the traveling point vortices}
  The result in Theorem \ref{thmE} also provides a family of solutions that desingularize the traveling point vortices for the gSQG equation. Indeed,  taking the transformation $\hat \omega_{tr,\ep}(x)=\ep^{-2}\omega_{tr,\ep}(\ep^{-1}x)$, we conclude from Theorem \ref{thmE} immediately that
   \begin{corollary}\label{corE}
   	Let $0<s<1$. Suppose that $f$ is a function satisfying $(\mathbf{H_{1}}) $ and $(\mathbf{H_{2}}) $ and $f\in C^{1-2s}$ if $0<s<\frac{1}{2}$. Then there is a constant $\ep_0>0$ small such that for any $\ep\in (0, \ep_0)$, \eqref{1-1} has a traveling-wave solution of the form $\theta_\ep(x,t)=\hat\omega_{tr,\ep}(x+ Wt\mathbf{e}_2 )$ for some function $\hat \omega_{tr,\ep}\in L^\infty(\mathbb{R}^2)$ in the sense that $\hat \omega_{tr,\ep}$ solves \eqref{1-3} with $\mathcal W=W$. Moreover, $\hat\omega_{tr,\ep}$ has the following properties:
   	\begin{itemize}
   		\item[(i)] $\hat\omega_{tr,\ep}$ is odd in $x_1$ and even in $x_2$. That is,
   		$$\hat\omega_{tr,\ep}(-x_1,x_2)=-\hat\omega_{tr,\ep}(x_1,x_2),\ \ \hat\omega_{tr,\ep}(x_1,-x_2)= \hat\omega_{tr,\ep}(x_1,x_2),\ \ \forall\ x\in\mathbb{R}^2;$$
   		\item[(ii)] It holds $$\hat\omega_{tr,\ep}=f(\ep^{2-2s}(\mathcal G_s\hat\omega_{tr,\ep}-Wx_1)-\mu_\ep),\ \ \text{in}\ \mathbb{R}_+^2, $$
   		for some constant $\mu_\ep$;
   		\item[(iii)] There holds in the sense of measure $$\hat\omega_{tr,\ep}(x)\rightharpoonup \kappa\mathbf{\delta}(x-(d_0,0))-\kappa\mathbf{\delta}(x+(d_0,0)),$$
   		 where  $d_0=\left(\frac{(1-s)c_s\kappa}{2^{2-2s} W}\right)^{\frac{1}{3-2s}}.$
   		Furthermore, there is a constant $R>0$ independent of $\ep$ such that $\text{spt}(\hat\omega_\ep)$ is contained in $B((d_0,0), R\ep)\cup B((-d_0,0), R\ep)$.
   	\end{itemize}
   \end{corollary}
	Corollary \ref{corE} (iii) implies that $\{\hat\omega_{tr,\ep}\}_{\ep\in(0,\ep_0)}$ is a sequence of regular solutions approximating the traveling point vortex pair for the gSQG equations.
	
	In \cite{Asqg}, Ao et al. constructed a family of solutions closed to the points vortices of the gSQG equation with the profile function $f(t)=t_+^p$ for $p\in(1, \frac{1+s}{1-s})$ by the Lyapunov-Schmidt reduction method. It can be seen that our result Corollary \ref{corE} covers the remaining case $p\in (0,1]$ for $\frac{1}{2}\leq s<1$.
	
	In the recent paper \cite{CQZZ}, a family of traveling solutions  for the gSQG equations with $\frac{1}{2}\leq s<1$ were constructed by the variational method. The solutions $\{\tilde\omega_{tr,\ep}\}$ obtained in \cite{CQZZ} solve the integral equation
	\begin{equation*}
		 \tilde\omega_{tr,\ep}=g( \mathcal G_s \tilde\omega_{tr,\ep}-Wx_1 -\tilde\mu_\ep),\ \ \text{in}\ \mathbb{R}^2_+,
	\end{equation*}
		for some constant $\tilde\mu_\ep$ and bounded non-decreasing function $g$. It is obvious that Corollary \ref{corE} can not be deduced from the result in  \cite{CQZZ}, since the profile function $f(\ep^{2-2s} \cdot )$ in Corollary \ref{corE} (ii) varies along with $\ep$ and is allowed to be unbounded. Therefore,  Corollary \ref{corE} provides a  new family of traveling-wave solutions for the gSQG equations.
	
	The paper is organized as follows.  In Section \ref{sec2}, for a large class of $f$, we construct traveling-wave solutions for the gSQG equation via a variational method. We first study the properties of maximizers of a limiting problem.  Based on these properties, we are able to construct traveling-wave solutions with small traveling speeds by maximizing the energy functional $E_\ep$. Then, we  study the asymptotic behavior of the  maximizers carefully in several steps.   With detailed asymptotic behaviors in hand, we prove the uniqueness of maximizers in Section \ref{sec3}.  Section \ref{sec4} is devoted to investigating   nonlinear stability.  We first prove a general   orbital stability theorem for the set of maximizers based on a combination of the variational method and the concentrated compactness lemma of Lions \cite{L84}. Then, we investigate the asymptotic behavior of maximizers in the rearrangement class and obtain the  orbital stability   Theorem \ref{thmS} by using the uniqueness result in Theorem \ref{thmU}.

\section{Proofs of Theorem 1.1 and Theorem 1.5}\label{sec2}
	In this section, we first consider the maximization problem \eqref{max**} and prove Theorem \ref{thmax*}.  Theorem \ref{thmE} follows immediately.
	
We assume that  $J: [0, +\infty) \to [0, +\infty)$ satisfies
	\begin{itemize}
		\item[$(\mathbf{H_{1}'}).$] $J$ is strictly convex and nonnegative;
		\item[$(\mathbf{H_{2}'}).$] $\lim_{t\to 0_+} J'(t)t^{s-1}=0$ and  $\liminf_{t\to +\infty} J'(t)t^{s-1}\geq K$.
	\end{itemize}
	Here $K$ is a large constant, which will be determined later. Note that if  $J(t)=\int_0^t f^{-1}(\tau) d\tau$ for some $f$ satisfying $(\mathbf{H_{1}})$ and $(\mathbf{H_{2}})$, then one can check that $J$ satisfies $(\mathbf{H_{1}'})$ and $(\mathbf{H_{2}'})$.

	Let $E_\ep(\omega)$ and $\mathcal{A}_\ep$ be defined as in \eqref{energy} and \eqref{constraint}. To obtain the existence of maximizers for \eqref{max**} we need to consider its limiting problem first.
	
	\subsection{The limiting problem}\label{sec2-1}
	We start with definitions of the energy functional and set of constraints for the limiting problem corresponding to \eqref{max**}. The energy functional associated with $E_\ep$ is
	\begin{equation*}
		E_0(\omega):=\frac{c_s}{2}\int_{\mathbb{R}^2}\int_{\mathbb{R}^2}\frac{ \omega(x)\omega(y)}{|x-y|^{2-2s}} dxdy -\int_{\mathbb{R}^2 } J(\omega(x)) dx,
	\end{equation*}
	and the constraint associated with $\mathcal{A}_\ep$ is
	$$\mathcal{A}_0:=\left\{\omega\in L^1 \cap L^{2-s}(\mathbb{R}^2 )\mid  \omega\geq 0, \ \ \  \int_{\mathbb{R}^2 } \omega(x) dx=\kappa\right\}.$$
	The limiting maximization problem associated with \eqref{max**} is
	
	 \begin{equation}\label{limitmax**}
		e_0:=\sup_{\omega\in \mathcal{A}_0} E_0(\omega).
	\end{equation}

	In the classical paper \cite{L84}, under a bit weaker assumption   $\lim_{t\to 0_+} J(t)t^{-1}=0$, Lions showed the existence of maximizers of $E_0$ over $\mathcal{A}_0$ (see Theorem \uppercase\expandafter{\romannumeral2}.2 and Corollary \uppercase\expandafter{\romannumeral2}.1 in \cite{L84}). As we shall see later, our assumption $\lim_{t\to 0_+} J'(t)t^{s-1}=0$ ensures that every maximizer is compactly supported, which is an essential property used in the next subsection (for similar results on rotating stars, we refer to \cite{Auc, Li91, McC}).
	
	In what follows, we will investigate some essential properties of maximizers under our hypotheses $(\mathbf{H_{1}'})$ and $(\mathbf{H_{2}'}).$
	
	Recall that $\mathcal{G}_s \omega:=\frac{c_s}{|x|^{2-2s}}\ast \omega$. Denote $\| \cdot\|_p:=\|\cdot\|_{L^p(\mathbb{R}^2)}$ for simplicity. The following two lemmas concerning convolution inequalities are needed in our later discussion.
	\begin{lemma}\label{lem2-1}
		Assume that $\omega\in L^1\cap L^{p}(\mathbb{R}^2 )$ for some $p>1$. If $1<p\leq s^{-1}$, then $\mathcal{G}_s \omega \in L^q(\mathbb{R}^2 ) $ for any $\frac{1}{1-s}<q<\frac{p}{1-sp}$ and for some constants $0<a,b<1$,
		\begin{equation}\label{2-3}
			\|\mathcal{G}_s \omega\|_q\leq C \left(\|   \omega\|_1^a\|   \omega\|_p^{1-a} +\|   \omega\|_1^b\|   \omega\|_p^{1-b}\right).
		\end{equation}

		If $ p> s^{-1}$, then \eqref{2-3} holds with $q=\infty$.
	\end{lemma}
	\begin{proof}
		We first consider the case  $1<p\leq s^{-1}$. We split the function $|x|^{ 2s-2}$ into two parts: $|x|^{ 2s-2}=|x|^{ 2s-2} 1_{\{|x|<1\}}+|x|^{ 2s-2} 1_{\{|x|\geq 1\}}$.
		It is easy to check that $|x|^{ 2s-2} 1_{\{|x|<1\}}\in L^r, \     \forall \ 1\leq r <\frac{1}{1-s}$ and $|x|^{ 2s-2} 1_{\{|x|\geq 1\}}\in L^r, \     \forall  \ r  >\frac{1}{1-s}$. Suppose $\frac{1}{1-s}<q<\frac{p}{1-sp}$, then there exist $1\leq r_1<\frac{1}{1-s}$, $r_2  >\frac{1}{1-s}$ and $1<p_1,p_2<p$ such that  $1+q^{-1}=r_1^{-1}+p_1^{-1}$ and $1+q^{-1}=r_2^{-1}+p_2^{-1}$.
		
		Then it remains to apply the following Young inequality   $$\| f \ast g\|_r\leq \|f\|_u \|g\|_v,\ \ \ \forall f\in L^u, \  g\in L^v, $$
		for $1\leq r, u, v \leq +\infty  \   \text{such that } 1+r^{-1}=u^{-1}+v^{-1},$ and the interpolation inequality
		\begin{equation}\label{2-4}
			\| f  \|_r\leq \|f\|_u^a \|f\|_v^{1-a},\ \ \ \forall f\in L^u\cap L^v,
		\end{equation}
		with $1\leq u<r<v \leq +\infty$ and $a=(r^{-1}-v^{-1})/(u^{-1}-v^{-1})\in (0,1)$.
		
		For $p> s^{-1}$, the proof is similar, so we will omit the detail  and finish the proof.
	\end{proof}
	
	\begin{lemma}\label{lem2-2}
		Suppose that $\omega\in L^1\cap L^{2-s}(\mathbb{R}^2 )$, then it holds
		\begin{equation}\label{2-5}
			\Big|\int_{\mathbb{R}^2}  \omega(x)\mathcal{G}_s \omega(x) dx  \Big| \leq C \|\omega\|_1^s \int_{\mathbb{R}^2} |\omega|^{2-s}.
		\end{equation}
	\end{lemma}
	\begin{proof}
		The well-known Hardy-Littlewood-Sobolev inequality states that
		\begin{equation}\label{2-6}
			\| \mathcal{G}_s \omega\|_q \leq C \|   \omega\|_p, \ \ \ \forall 1<p<q<+\infty \ \text{with} \ \frac{1}{q} =\frac{1}{p}-s.
		\end{equation}
		
		Applying H\"older's inequality, \eqref{2-6} with $q=\frac{2-s}{1-s}$ and the interpolation inequality \eqref{2-4}, for any $\omega_1, \omega_2\in  L^1\cap L^{2-s}(\mathbb{R}^2 )$ we obtain
		\begin{equation}\label{2-6-1}
			\Big|\int_{\mathbb{R}^2}  \omega_1(x)\mathcal{G}_s \omega_2(x) dx  \Big| \leq \|\omega_1\|_{2-s}\|\mathcal{G}_s \omega_2\|_{\frac{2-s}{1-s}}\leq C \|\omega_1\|_{2-s} \|\omega_2\|_{2-s}^{1-s}\|\omega_2\|_{1}^s,
		\end{equation}
		which implies \eqref{2-5} by taking $\omega_1=\omega_2=\omega$ and completes the proof.
	\end{proof}
	
	We first show the radial symmetry and derive the Euler-Lagrange equation for a maximizer of \eqref{limitmax**}.
	\begin{lemma}\label{lem2-3}
		Let   $\omega_0 \in \mathcal{A}_0$ be a maximizer of $E_0$ over $\mathcal{A}_0$. Then $\omega_0$ must be radially symmetric and non-increasing with respect to some point. Moreover,  there exists a constant $\mu_0\in \mathbb{R}$ such that
		\begin{equation}\label{2-7}
			\begin{cases}
				\mathcal{G}_s\omega_0-J'(\omega_0) \leq  \mu_0,\quad &\text{on}\ \ \{\omega_0=0\},\\
				\mathcal{G}_s\omega_0 - J'(\omega_0) =  \mu_0,\quad &\text{on}\ \ \{ \omega_0>0\}.
			\end{cases}	
		\end{equation}
	\end{lemma}
	\begin{proof}
		The radial symmetry and monotonicity of $\omega_0$ are easy consequences of the strict rearrangement inequality (see Theorems 3.7 and 3.9 in \cite{Lie}).
		
		Note that for $\delta>0$ small, the set $\{\omega_0>\delta\}\not=\emptyset$ due to $\int_{\mathbb{R}^2} \omega_0=\kappa>0$. We fix a $\delta>0$ such that $\{\omega_0>\delta\}\not=\emptyset$ and take a function $\phi_0$ so that $\int_{\mathbb{R}^2} \phi_0=1$ and $\text{spt}(\phi_0)\subset \{\omega_0>\delta\}$. For any function $\phi $ bounded from below such that $\phi\geq 0$ on the set $\{\omega_0\leq\delta\}$, we take a  family of test functions as follows:
		$$\omega^t:=\omega_0+t(\phi-\phi_0\int_{\mathbb{R}^2}\phi),$$
		which belong  to $\mathcal{A}_0$ for $|t|$ small.
		Since $\omega_0$ is a maximizer, we have
		\begin{equation*}
			0= \frac{d E_0(\omega^t)}{d t} \Bigg{|}_{t=0}=\int_{\mathbb{R}^2}  (\mathcal{G}_s\omega_0-J'(\omega_0)-\mu_0)\phi  dx,
		\end{equation*}
		where $\mu_0:=\int_{\mathbb{R}^2}  (\mathcal{G}_s\omega_0-J'(\omega_0))\phi_0$.
		Then \eqref{2-7} follows from the arbitrariness of $\phi$ and hence the proof is   complete.
	\end{proof}
	
	Denote
	\begin{equation}\label{I0}
	I_0:=\sup_{\omega\in \mathcal{A}_0} E_0(\omega).
	\end{equation}
	Then $I_0<+\infty$ by \cite{L84}.
	\begin{lemma}\label{lem2-4}
		There is a constant  $c_0 >0$ such that if $\omega_0\in \mathcal{A}_0 $ is a maximizer of $E_0$ over $\mathcal{A}_0 $, then one has
		\begin{equation}\label{2-8}
			\|\omega_0\|_{2-s}\leq c_0.
		\end{equation}
	\end{lemma}
	\begin{proof}
		Take the constant $K$ in  the hypothesis $(\mathbf{H_{2}'})$ as $K=C \kappa^s+2$, where $C$ is the constant in \eqref{2-5}. Then there is a constant  $t_0>0$ such that $J(t)> (C \kappa^s+1) t^{2-s}$ for $t>t_0$.   On the one hand,  by the definition of $E_0$ and \eqref{2-5}, we deduce
		\begin{equation*}
			\int_{\mathbb{R}^2} J(\omega_0)\leq -I_0 +C\kappa^s \int_{\mathbb{R}^2}(\omega_0)^{2-s} dx.
		\end{equation*}
		On the other hand, we infer from the choice of $t_0$ that
		\begin{align*}
			(C\kappa^s+1) \int_{\mathbb{R}^2}(\omega_0)^{2-s} dx&= (C\kappa^s+1) \int_{ \{ \omega_0\leq t_0\}}(\omega_0)^{2-s} dx +(C\kappa^s+1) \int_{ \{ \omega_0> t_0\}}(\omega_0)^{2-s} dx\\
			&\leq (C\kappa^s+1)t_0^{1-s}\kappa+\int_{\mathbb{R}^2} J(\omega_0).
		\end{align*}
		Therefore, we arrive at \eqref{2-8} by taking $c_0= ( -I_0+(C\kappa^s+1)t_0^{1-s}\kappa)^{\frac{1}{2-s}}.$
	\end{proof}
	
	\begin{lemma}\label{lem2-5}
		Let $\omega_0\in \mathcal{A}_0 $ be a maximizer and $\mu_0$ be the constant defined in Lemma \ref{lem2-3}, then  one has
		\begin{equation}\label{2-9}
			\mu_0>0.
		\end{equation}
	\end{lemma}
	\begin{proof}
		Take $r_0>0$ such that $\int_{B(0,r_0)} \omega_0\geq\frac{\kappa}{2}$. Since for any $r\geq 0$, $\int_{B(0,r)} \omega_0\leq \kappa$, there is a point $x^r\in B(0,r)$ such that $\omega_0(x^r)\leq \pi^{-1}\kappa r^{-2}.$ Then we infer from the hypothesis $(\mathbf{H_{2}'})$ that
		$$J'(\omega_0(x^r))=o(r^{2s-2}), \ \  \text{as}\ r\to+\infty.$$
		On the other hand, we have
		$$\mathcal{G}_s\omega_0(x^r)\geq \frac{c_s}{(2r)^{2-2s}} \int_{B(0,r)} \omega_0\geq 4^{s-2}c_s\kappa r^{2s-2},\  \  \forall \ r\geq r_0.$$
		Thus, we derive from \eqref{2-7} that
		$$\mu_0\geq \mathcal{G}_s\omega_0(x^r)-J'(\omega_0(x^r))\geq (4^{s-2}c_s\kappa+o(1)) r^{2s-2}>0,$$ for $r$ sufficiently large and hence the proof is finished.
	\end{proof}
	
	Lemma \ref{lem2-5} will give a uniform bound for all maximizers.
	\begin{corollary}\label{lem2-6}
		There is a constant  $ c_1>0$ such that if $\omega_0\in \mathcal{A}_0 $ is a maximizer, then
		\begin{equation}\label{2-10}
			\|\omega_0\|_{\infty}\leq c_1.
		\end{equation}
	\end{corollary}
	\begin{proof}
		By   Lemmas  \ref{lem2-3} and \ref{lem2-5},   we have
		$$J'(\omega_0)= \mathcal{G}_s\omega_0-\mu_0\leq\mathcal{G}_s\omega_0  , \ \ \  \text{on}\ \{\omega_0>0\}.$$
		Let $p_1=2-s$. Then we derive from Lemmas \ref{lem2-1} and \ref{lem2-4} that $ \|\mathcal{G}_s\omega_0\|_q\leq C$ for $\frac{1}{1-s}<q< \frac{p_1}{1-sp_1}$. Using the fact the $J'(t)\geq \frac{K}{2} t^{1-s}$ for $t>t_1$ due to the hypothesis $(\mathbf{H_{2}'})$ with $t_1$   a large constant, we obtain $\|\omega_0\|_r\leq C$ for $1\leq r<p_2$  with $p_2:=\frac{(1-s)p_1}{ 1-sp_1 }$. Notice that $p_2= p_1 +\frac{s p_1(p_1-1)}{ 1-sp_1 }$ and $\frac{s p_1(p_1-1)}{ 1-sp_1 }>0$ is increasing in $p_1\in (1,s^{-1})$. So, using Lemma \ref{lem2-1}, a simple bootstrap argument will prove this lemma.
	\end{proof}
	
	\begin{lemma}\label{lem2-7}
		It holds $$I_0>0.$$
	\end{lemma}
	\begin{proof}
		We take a   function $\rho:=1_{B(0,\sqrt{\kappa/\pi})}$ and define $\rho_r(x):=(  r^{-2})\rho( r^{-1} x)$. It can be seen that $\rho_r\in \mathcal{A}_0$.
		
		By the hypothesis $(\mathbf{H_{2}'})$, we find
		$$\int_{\mathbb{R}^2} J(\rho_r)=o(1) \int_{\mathbb{R}^2}  (\rho_r)^{2-s}=  o(r^{2s-2}), \ \ \ \text{as}\ r\to+\infty. $$
		On the other hand, a change of variables gives
		$$\int_{\mathbb{R}^2} \rho_r \mathcal{G}_s \rho_r=  r^{2s-2}\int_{\mathbb{R}^2} \rho  \mathcal{G}_s \rho.$$
		Thus, we can take a constant $r_1$ sufficiently large such that $$I_0\geq E_0(\rho_{r_1}) \geq c_3 r_1^{2s-2}>0.$$
	\end{proof}
	
	\begin{corollary}\label{lem2-8}
		If $\omega_0\in \mathcal{A}_0 $ is a maximizer of $E_0$ over $\mathcal{A}_0$, then  it holds
		\begin{equation}\label{2-11}
			\|\mathcal{G}_s\omega_0\|_\infty \geq 2 I_0\kappa^{-1}.
		\end{equation}
	\end{corollary}
	\begin{proof}
		One has
		$$I_0 = E_0(\omega_0)\leq \frac{1}{2}\int_{\mathbb{R}^2} \omega_0 \mathcal{G}_s \omega_0 dx  \leq \frac{\kappa}{2} \|\mathcal{G}_s\omega_0\|_\infty,$$
		which implies \eqref{2-11} and completes the proof.
	\end{proof}
	
	\begin{lemma}\label{lem2-9}
		There is a constant  $\eta>0$ such that if $\omega_0\in \mathcal{A}_0 $ is a maximizer of $E_0$ over $\mathcal{A}_0$, then we have
		\begin{equation}\label{2-12}
			\sup_{x\in \mathbb{R}^2}\int_{|x-y|<1} \omega_0(y)dy \geq \eta>0.
		\end{equation}
	\end{lemma}
	\begin{proof}
		Denote $\eta_0:=\sup_{x\in \mathbb{R}^2}\int_{|x-y|<1} \omega_0(y)dy$. Let $r:=(c_s I_0^{-1}\kappa^2)^{\frac{1}{2-2s}}$. For any $x\in \mathbb{R}^2$, we calculate
		\begin{equation*}
			\begin{split}
				\mathcal{G}_s \omega_0(x)&=\int_{|x-y|<1} \frac{c_s\omega_0(y)}{|x-y|^{2-2s}} dy+\int_{1\leq |x-y|<r}\frac{c_s\omega_0(y)}{|x-y|^{2-2s}} dy+\int_{|x-y|\geq r} \frac{c_s\omega_0(y)}{|x-y|^{2-2s}} dy\\
				&=: A_1+A_2+A_3.
			\end{split}
		\end{equation*}
		For the first term $A_1$, we use Lemma  \ref{lem2-1} and Corollary \ref{lem2-6} to obtain
		$$A_1\leq C(\eta_0^a+\eta_0^b),\ \ \ \text{for some } \ a,b\in (0,1).$$
		For the second term $A_2$, noticing that the annulus $\{y \ | \ 1<|x-y|<r\}$ can be covered by $\leq Cr^2$ disks with radius $1$, we find $$A_2\leq C c_s r^2 \eta_0.$$
		Now, by the choice of $r$, the last integral $A_3$ can be estimated by $$A_3\leq c_s r^{2s-2} \kappa=I_0\kappa^{-1}.$$
		So, by the arbitrariness of $x$ and Corollary \ref{lem2-8}, we arrive at
		$$C(\eta_0^a+\eta_0^b+\eta_0)\geq I_0\kappa^{-1},$$ which implies \eqref{2-12} and completes the proof.
	\end{proof}
	
	Now we show that the Lagrange multiplier $\mu_0$ is uniformly bounded from below.
	\begin{lemma}\label{lem2-10}
		There is a constant  $\mu_*>0$ such that if $\omega_0\in \mathcal{A}_0 $ is a maximizer of $E_0$ over $\mathcal{A}_0$, then there holds
		\begin{equation}\label{2-13}
			\mu_0\geq \mu_*.
		\end{equation}
	\end{lemma}
	\begin{proof}
		By the previous lemma, one can find a point $x_\eta\in \mathbb{R}^2$ such that $$\int_{|x_\eta-y|<1} \omega_0(y)dy \geq \frac{\eta }{2}>0.$$ In view of Lemma \ref{lem2-3}, we may assume that $\omega_0$ is radially symmetric with respect to the origin and non-increasing. Thus, we get $$\int_{| y|<1} \omega_0(y)dy \geq \int_{|x_\eta-y|<1} \omega_0(y)dy\geq \frac{\eta}{2}.$$
		
		For any $r>1$ large, there is a point $x^r\in B(0,r)$ such that $\omega_0(x^r)\leq \pi^{-1}\kappa r^{-2}.$ Then we infer from the hypothesis $(\mathbf{H_{2}'})$ that
		$$J'(\omega_0(x^r))=o(r^{2s-2}), \ \  \text{as}\ r\to+\infty.$$
		On the other hand, we have
		$$\mathcal{G}_s\omega_0(x^r)\geq \frac{c_s}{(2r)^{2-2s}} \int_{B(0,1)} \omega_0\geq 4^{s-2}c_s \eta  r^{2s-2}.$$
		Thus, by \eqref{2-7}, we can take a large constant $r_0$  such that
		$$\mu_0\geq \mathcal{G}_s\omega_0(x^{r_0})-J'(\omega_0(x^{r_0}))\geq (4^{s-2}c_s \eta_0+o(1)) r_0^{2s-2}\geq 4^{s-3}c_s \eta_0r_0^{2s-2}=:\mu_*>0.$$
		The proof is therefore finished.
	\end{proof}

	\begin{lemma}\label{lem2-11}
		There exists a constant $R_*>0$ such that if $\omega_0\in \mathcal{A}_0 $ is a maximizer of $E_0$ over $\mathcal{A}_0 $, then the diameter of the support of $\omega_0$ is less than $2R_*$. That is,
		\begin{equation}\label{2-14}
			\text{diam}(\text{spt}(\omega_0))\leq 2 R_*.
		\end{equation}
	\end{lemma}
	\begin{proof}
		
		Without loss of generality, we may assume that $\omega_0$ is radially symmetric and non-increasing with respect to the origin. Since $\int \omega_0=\kappa<+\infty$ and $\omega_0(x)=\omega_0(|x|)$ is   non-increasing in $|x|$, we have $$\int_{|x-y|<1} \omega_0(y)\leq C|x|^{-2} $$ for $|x|$ large. Then, through similar calculations as the proof of Lemma \ref{lem2-9}, for large $|x|$ and any constant $r$, we get
		\begin{equation*}
			\begin{split}
				\mathcal{G}_s \omega_0(x)&=\int_{|x-y|<1} \frac{c_s\omega_0(y)}{|x-y|^{2-2s}} dy+\int_{1\leq |x-y|<r}\frac{c_s\omega_0(y)}{|x-y|^{2-2s}} dy+\int_{|x-y|\geq r} \frac{c_s\omega_0(y)}{|x-y|^{2-2s}} dy\\
				&\leq C(|x|^{-2a}+|x|^{-2b}+r^2|x|^{-2})+c_s\kappa r^{2s-2},
			\end{split}
		\end{equation*}
		which will tend  to $0$ if we first take $|x|\to+\infty$   then $r\to+\infty$. Fixed $r$ such that $c_s\kappa r^{2s-2}\leq \frac{\mu_*}{2}$. Then, one can find a constant $R_*$ such that $\mathcal{G}_s \omega_0(x)< \mu_* $ whenever $|x|>R_*$. Take $R_4=\max\{R_3, c_1, R_*+1\}$. We infer from Lemmas \ref{lem2-3} and \ref{lem2-10}  that $\omega_0(x)=0$ for any $|x|>R_*$. The proof of this lemma is hence completed.
	\end{proof}

	Next, we further study the properties of the maximizers in the special case $J(t)=L t^{1+\frac{1}{p}}$ for some constants $L>0$ and $p\in(0, \frac{1}{1-s})$. We first determine the Lagrange multiplier $\mu_0$. Recall that we denote $I_0=\sup_{\omega\in \mathcal{A}_0} E_0(\omega)\in(0,+\infty)$ to be maximum value of $E_0$.
	\begin{lemma}\label{lem2-12}
		Suppose that $J(t)=L t^{1+\frac{1}{p}}$ for some constants $L>0$ and $p\in(0,  \frac{1}{1-s})$. Let $\omega_0$ be a maximizer of $E_0$ over $\mathcal{A}_0$ with \eqref{2-7} for some $\mu_0$. Then, we have
		\begin{equation}\label{2-15}
			\mu_0\kappa =C_{s,p }I_0,
		\end{equation}
		for some constant $C_{s,p }$ depending only on $s,p $.
	\end{lemma}
	\begin{proof}
		Let $\gamma:=1+1/p$. We take a family of functions $(\omega_0)_t(x):=t^{-2}\omega_0(t^{-1}x)$. By changing of variables, we find
		$$E_0((\omega_0)_t)=\frac{t^{2s-2}}{2}\int_{\mathbb{R}^2} \omega_0\mathcal{G}_s \omega_0- t^{2-2\gamma} \int_{\mathbb{R}^2} J(\omega_0).$$
		Since $\omega_0$ is a maximizer, we have
		\begin{equation}\label{2-16}
			0=\frac{d}{dt}\bigg|_{t=1} E_0((\omega_0)_t)= (s-1)\int_{\mathbb{R}^2} \omega_0\mathcal{G}_s \omega_0-  (2-2\gamma)  \int_{\mathbb{R}^2} J(\omega_0).
		\end{equation}
		By \eqref{2-7}, one has
		$$L\gamma  \omega_0^{\gamma-1}=(\mathcal{G}_s\omega_0-\mu_0)_+.$$
		Multiplying the above equation by $\omega_0$ and integrating, we obtain
		\begin{equation}\label{2-17}
			\kappa\mu_0= \int_{\mathbb{R}^2} \omega_0\mathcal{G}_s \omega_0- \gamma\int_{\mathbb{R}^2} J(\omega_0).
		\end{equation}
		Then \eqref{2-15} follows from simply calculations by using  the definition of $E_0$, \eqref{2-16} and \eqref{2-17}. The constant $C_{s,p }=A_\gamma:=\frac{2-\gamma-s\gamma}{ 2-s-\gamma}$ with $\gamma =1+1/p$.
	\end{proof}

	Lemma \ref{lem2-12} states that the Lagrange multiplier is the same for all maximizers. Set $\psi_0:=\mathcal{G}_s \omega_0$. Then, $\psi_0$ is radially symmetric and  satisfies the  following equation by Lemma \ref{lem2-3}.
	\begin{equation}\label{2-18}
		\begin{cases}
			(-\Delta)^s\psi_0=\left(\frac{p}{L(p+1)}\right)^p (\psi_0-\mu_0)_+^p,\ \  \text{in}\ \mathbb{R}^2,\\
			\psi_0(x)\to0, \ \ \text{as} \  |x|\to+\infty.
		\end{cases}
	\end{equation}
	
	Thus, the uniqueness result in \cite{C20} is applicable for $p\in (1, \frac{1}{1-s})$.  Furthermore, Ao et al. \cite{Asqg} showed the non-degeneracy of the linearized equation for $p\in (1, \frac{1+s}{1-s})$. Note that for $p\in (0,1]$ the uniqueness of maximizers has been proved in \cite{DYY}.  Summarizing these results, we obtain
	\begin{proposition}\label{lem2-13}
		Suppose that $J(t)=L t^{1+\frac{1}{p}}$ for some constants $L>0$ and $p\in(0, \frac{1}{1-s})$. Then, up to translations $E_0$ has a unique maximizer $\omega_0$ over $\mathcal A_0$. Moreover, the following properties hold:
		\begin{itemize}
			\item[(i)] $\omega_0$ is compact supported and radially symmetric and decreasing about some point;
			\item[(ii)] $   \omega_0 =\left(\frac{p}{L(p+1)}\right)^p(\mathcal{G}_s\omega_0-\mu_0)_+^p$ for $\mu_0=\kappa^{-1} C_{s,p }E_0(\omega_0)>0$;
			\item[(iii)] If in addition $p\in (1, \frac{1+s}{1-s})$, then $\omega_0\in C^1$ and the kernal of the linearized operator $$\omega\mapsto \omega-p\left(\frac{p}{L(p+1)}\right)^p\left(\mathcal{G}_s\omega_0-\mu_0\right)_+^{p-1}\mathcal{G}_s\omega$$ in the space $L^1\cap L^\infty(\mathbb{R}^2)$ is $$\text{span}\{  \partial_{x_1}\omega_0, \partial_{x_2}\omega_0\}.$$
		\end{itemize}
	\end{proposition}
	\begin{proof}
		The existence of a maximizer $\omega_0$, (i) and (ii) follow  from the above lemmas. For   uniqueness, we refer to \cite{C20, DYY}. The non-degeneracy   was proved in Proposition 3.2 in \cite{Asqg} in terms of $\psi_0=\mathcal{G}_s\omega_0$, form which one can obtain (iii) easily.
	\end{proof}
	\subsection{Existence of traveling-wave solutions via maximization}\label{sec2-2}
	
	In this subsection,	we will obtain the existence of traveling-wave solutions by considering the maximization problem \eqref{max**}, whose associated limiting problem has been studied in the preceding subsection.

	\subsubsection{Existence of maximizers}\label{sec2-2-1}
	 To obtain the compactness of maximizing sequence, we first maximize 	$E_\ep$ over a set smaller than $\mathcal{A}_{\ep}$ given by
	$$\mathcal{A}_{\ep, \Gamma}:=\Big\{\omega\in L^1\cap L^{\infty}(\mathbb{R}^2_+)\ \Big|\ 0\leq \omega \leq \Gamma ,\ \ \ \text{spt}(\omega)\subset B(\ep^{-1}(d_0,0), \ep^{-1} d_0/2), \  \  \  \int_{\mathbb{R}^2_+} \omega(x) dx=\kappa\Big\},$$
	where $\Gamma>0$ is a number that will be fixed later.
	\begin{lemma}\label{lem2-14}
		For given $\ep, \Gamma>0$, there exists a  function $\omega_{\ep, \Gamma}\in \mathcal{A}_{\ep, \Gamma}$ such that $$ E_\ep(\omega_{\ep, \Gamma})=\sup_{\omega\in \mathcal{A}_{\ep, \Gamma}} E_\ep(\omega).$$
	\end{lemma}
	\begin{proof}
		Let $\{\omega_j\}_{j=1}^\infty\subset \mathcal{A}_{\ep, \Gamma}$ be a maximizing sequence. By the definition of $\mathcal{A}_{\ep, \Gamma}$, we know that $\{\omega_j\}_{j=1}^\infty$ is uniformly bounded in $L^1 \cap L^\infty(B(\ep^{-1}(d_0,0), \ep^{-1} d_0/2))$.     Passing to a subsequence (still denoted by $\{\omega_j\}_{j=1}^\infty$), we may assume $\omega_j \rightarrow \omega_{\ep, \Gamma}$ weakly star in  $L^\infty$. By the weak star convergence, it holds $$0\leq \omega_{\ep, \Gamma}\leq \Gamma,\ \ \ \int_{\mathbb{R}^2_+ } \omega_{\ep, \Gamma}(x) dx=\lim_{j\to+\infty}\int_{\mathbb{R}^2_+  } \omega_j(x)dx=\kappa.$$   Note that $G_s(x,y)\in L^{r}(B(\ep^{-1}(d_0,0), \ep^{-1} d_0/2)\times B(\ep^{-1}(d_0,0), \ep^{-1} d_0/2))$ for $1\leq r<\frac{1}{1-s}$. We have $$\lim_{j\to+\infty} \frac{c_s}{2}\int_{\mathbb{R}^2_+ }\int_{\mathbb{R}^2_+ }\frac{ \omega_j(x)\omega_j(y)}{|x-y|^{2-2s}} dxdy=\frac{c_s}{2}\int_{\mathbb{R}^2_+ }\int_{\mathbb{R}^2_+ }\frac{ \omega_{\ep, \Gamma}(x)\omega_{\ep, \Gamma}(y)}{|x-y|^{2-2s}} dxdy.$$
		On the other hand, since $J$ is convex, by the lower semi-continuity, we find $$\lim_{j\to+\infty} \int_{\mathbb{R}^2_+  } J(\omega_j(x)) dx\geq \int_{\mathbb{R}^2_+  } J(\omega_{\ep, \Gamma}(x)) dx.$$ Therefore, $$\sup_{\omega\in \mathcal{A}_{\ep, \Gamma}} E_\ep(\omega)=\lim_{j\to+\infty} E_\ep(\omega_j)\leq E_\ep(\omega_{\ep, \Gamma})\leq \sup_{\omega\in \mathcal{A}_{\ep, \Gamma}} E_\ep(\omega), $$ which implies that $\omega_{\ep, \Gamma}$ is a maximizer and completes the proof.
	\end{proof}

	For a non-negative function $\zeta$, we shall say that $\zeta$ is Steiner symmetric in the $x_2$-variable  if for any fixed $x_1$, $\zeta$  is the unique even function of $x_2$  such that
	\begin{equation*}
		\zeta(x_1, x_2)>\tau\ \ \ \text{if and only if}\ \ \ |x_2|<\frac{1}{2}\,|\left\{y_2\in \mathbb{R} \mid\ \zeta(x_1,y_2)>\tau \right\}|_{\mathbb R},
	\end{equation*}
	where $|\cdot|_{\mathbb R}$ denotes the Lebesgue measure on $\mathbb R$.
	
	For a function $0\leq \zeta\in L^1(D)$ with $D$ a domain symmetric with respect to $x_1$-axis, we denote by $\zeta^\star$ the Steiner symmetrization of $\zeta$, which is the unique  function in the rearrangement class that is Steiner symmetric in the $x_2$-variable (see \cite{Lie} for more details about rearrangement). A key fact about the Steiner symmetrization is the rearrangement inequality (see e.g. Theorems 3.7 and 3.9 in  \cite{Lie})
	\begin{equation*}
		\int \zeta^\star\mathcal{G}_s\zeta^\star \geq \int \zeta \mathcal{G}_s\zeta,\ \ \int \zeta^\star\mathcal{G}_s^+\zeta^\star \geq \int \zeta \mathcal{G}_s^+\zeta,
	\end{equation*}
	with strict inequality unless $\zeta(\cdot)\equiv\zeta^\star(\cdot+(0,c))$ for some $c\in\mathbb{R}$.
	
	\begin{lemma}\label{lem2-15}
		For given $\ep, \Gamma>0$, let $\omega_{\ep, \Gamma}\in \mathcal{A}_{\ep, \Gamma}$ be a maximizer of  $E_\ep$ over $\mathcal{A}_{\ep, \Gamma}$. Then  $\omega_{\ep, \Gamma}$ is symmetric non-increasing with respect to some line $\{x_2=const.\}$. Moreover, there exists a constant $\mu_{\ep,\Gamma}\in \mathbb{R}$ such that
		\begin{equation}\label{2-19}
			\begin{cases}
				\mathcal{G}_s^+\omega_{\ep, \Gamma}- W\ep^{3-2s}x_1-J'(\omega_{\ep, \Gamma}) \leq  \mu_{\ep, \Gamma},\quad &\text{on}\ \ \{\omega_{\ep, \Gamma}=0\},\\
				\mathcal{G}_s^+\omega_{\ep, \Gamma}- W\ep^{3-2s}x_1-J'(\omega_{\ep, \Gamma}) =  \mu_{\ep, \Gamma},\quad &\text{on}\ \ \{0<\omega_{\ep, \Gamma}<\Gamma\},\\
				\mathcal{G}_s^+\omega_{\ep, \Gamma}- W\ep^{3-2s}x_1-J'(\omega_{\ep, \Gamma}) \geq  \mu_{\ep, \Gamma},\quad &\text{on}\ \ \{\omega_{\ep, \Gamma}=\Gamma\}.
			\end{cases}	
		\end{equation}
	\end{lemma}	
	\begin{proof}
		The   symmetry and monotonicity of $\omega_{\ep, \Gamma}$ with respect to some line $\{x_2=const.\}$ is an easy consequence of the strict rearrangement inequality.
		
		For any $\omega\in \mathcal{A}_{\ep, \Gamma}$, we take a  family of test functions as follows:
		$$\omega_t:=\omega_{\ep, \Gamma}+t(\omega-\omega_{\ep, \Gamma}),\ \ \  t\in[0,1].$$
		Since $\omega_{\ep, \Gamma}$ is a maximizer, we have
		\begin{equation*}
			0\geq \frac{d E_\ep(\omega_t)}{d t} \Bigg{|}_{t=0_+}=\int_{\mathbb{R}^2_+ }  (\mathcal{G}_s^+\omega_{\ep, \Gamma}- W\ep^{3-2s}x_1-J'(\omega_{\ep, \Gamma}))(\omega-\omega_{\ep, \Gamma})  dx.
		\end{equation*}
		That is,
		\begin{equation*}
			\int_{\mathbb{R}^2_+ }  (\mathcal{G}_s^+\omega_{\ep, \Gamma}- W\ep^{3-2s}x_1-J'(\omega_{\ep, \Gamma})) \omega   dx\geq \int_{\mathbb{R}^2_+ }  (\mathcal{G}_s^+\omega_{\ep, \Gamma}- W\ep^{3-2s}x_1-J'(\omega_{\ep, \Gamma})) \omega_{\ep, \Gamma}   dx.
		\end{equation*}
		Then \eqref{2-19} follows by applying an adaption of the bathtub principle (see section 1.14 in \cite{Lie}). The proof is thus complete.
	\end{proof}
	
	By the definition of  $I_0$ (see \eqref{I0}) in the previous subsection we have
	\begin{lemma}\label{lem2-16}
		There are two constants $\ep_0, \Gamma_0>0$ such that for any $\ep\in (0,\ep_0)$ and $ \Gamma>\Gamma_0$, it holds
		\begin{equation}\label{2-20}
			\sup_{\omega\in \mathcal{A}_{\ep, \Gamma}} E_\ep(\omega)\geq I_0+O(\ep^{2-2s}), \ \ \ \text{as}\ \ep\to0.
		\end{equation}
	\end{lemma}	
	\begin{proof}
		Let $\omega_0$ be the maximizer of $E_0$ over $\mathcal{A}_0$. We may assume that $\omega_0$ is symmetric and non-increasing with respect to the origin. Since $\omega_0$ has compact support and is bounded, we may take $\Gamma_0>0$ large and $\ep_0>0$ small such that $$\bar \omega_0(x):=\omega_0(x-\ep^{-1}(d_0,0))\in \mathcal{A}_{\ep, \Gamma}.$$ Then direct computation shows
		$$\sup_{\omega\in \mathcal{A}_{\ep,\Gamma}} E_\ep(\omega)\geq E_\ep(\bar\omega_0)=E_0(\omega_0)+O(\ep^{2-2s}),$$
		which implies the desired estimate and finishes the proof.
	\end{proof}	
	
	\begin{lemma}\label{lem2-17}
		There are two constants $C_0, \ep_1>0$ such that if  $\omega_{\ep, \Gamma}\in \mathcal{A}_{\ep, \Gamma}$ is a maximizer of  $E_\ep$ over $\mathcal{A}_{\ep, \Gamma}$ for $\ep\in (0,\ep_1)$ and $\Gamma>\Gamma_0$, then it holds
		\begin{equation}\label{2-21}
			\|\omega_{\ep, \Gamma}\|_{2-s}\leq C_0.
		\end{equation}
	\end{lemma}	
	\begin{proof}
		Notice that by Lemma \ref{lem2-16}, there exists $0<\ep_1\leq \ep_0$ such that $E_\ep(\omega_{\ep, \Gamma})\geq \frac{I_0}{2}>0$ for $\ep\in(0,\ep_1)$ and $\Gamma>\Gamma_0$. Thus, we have $$\int_{\mathbb{R}^2_+ } J(\omega_{\ep, \Gamma})\leq \frac{1}{2}\int_{\mathbb{R}^2_+ } \omega_{\ep, \Gamma}\mathcal{G}_s \omega_{\ep, \Gamma}-E_\ep(\omega_{\ep, \Gamma})<\frac{1}{2}\int_{\mathbb{R}^2_+ } \omega_{\ep, \Gamma}\mathcal{G}_s \omega_{\ep, \Gamma},$$ from which, by Lemma \ref{lem2-2} and a similar argument as the proof of Lemma \ref{lem2-4}, we obtain the estimate \eqref{2-21} and complete the proof.
	\end{proof}	
	
	\begin{lemma}\label{lem2-18}
		There is a constant $ \ep_2>0$ such that if  $\omega_{\ep, \Gamma}\in \mathcal{A}_{\ep, \Gamma}$ is a maximizer of  $E_\ep$ over $\mathcal{A}_{\ep, \Gamma}$ for $\ep\in (0,\ep_2)$ and $\Gamma>\Gamma_0$, then it holds
		\begin{equation}\label{2-22}
			\mu_{\ep, \Gamma}>0.
		\end{equation}
	\end{lemma}	
	\begin{proof}
		For fixed $\Gamma>\Gamma_0$, we see from Lemma \ref{lem2-16} that $\{\omega_{\ep,\Gamma}\}_{\ep\in(0, \ep_0)}$ is  a maximizing sequence of $E_0$. Then   Theorem \uppercase\expandafter{\romannumeral2}.2 and Corollary \uppercase\expandafter{\romannumeral2}.1 in \cite{L84}  give  a subsequence (still denoted by $\{\omega_{\ep,\Gamma}\} $ for convenience) converges to a maximizer $\omega_0$ of $E_0$ in $L^1\cap L^{2-s}(\mathbb{R}^2)$ after suitable translations. Thus, for $\ep$ small , we have $\int_{B(p_\ep, R_*)} \omega_{\ep,\Gamma} \geq \frac{\kappa}{2}$ for some point $p_\ep$. Here, $R_*$ is the constant in Lemma \ref{lem2-11}.
		Since $\int \omega_{\ep, \Gamma}=\kappa$, for $R_*<r <\frac{d_0}{2\ep}$ large, we can take a point $x^r\in B(p_\ep, r)\cap B(\ep^{-1}(d_0,0), \ep^{-1} d_0/2)$ such that $\omega_{\ep, \Gamma}(x^r)\leq C r^{-2}$. Then, using \eqref{2-19} and the hypothesis  $(\mathbf{H_{2}'})$, by similar calculations as the proof of Lemma \ref{lem2-10}, we have
		$$\mu_{\ep, \Gamma}\geq \mathcal{G}_s^+\omega_{\ep, \Gamma}- W\ep^{3-2s}x_1-J'(\omega_{\ep, \Gamma})\geq (4^{s-2}c_s\kappa+o(1))r^{2s-2}-  O(\ep^{2-2s}),$$
		which implies \eqref{2-22} by taking $r=\ep^{-\frac{1}{2}}$ and $\ep$ sufficiently small.
	\end{proof}		
	
	As an immediate consequence of Lemmas \ref{lem2-15}, \ref{lem2-17}	and \ref{lem2-18}, we can obtain the following result through a similar argument as the proof of Corollary \ref{lem2-6}. We leave the details of the proof to   readers.
	\begin{corollary}\label{lem2-19}
		There is a constant  $ C_1>0$ such that if $\omega_{\ep, \Gamma}\in \mathcal{A}_{\ep, \Gamma}$ is a maximizer of  $E_\ep$ over $\mathcal{A}_{\ep, \Gamma}$ for $\ep\in (0,\ep_2)$ and $\Gamma>\Gamma_0$, then
		\begin{equation}\label{2-24}
			\|\omega_{\ep, \Gamma}\|_{\infty}\leq C_1.
		\end{equation}
		Moreover, if  $\Gamma> \max\{\Gamma_0, C_1\}$, then
		\begin{equation}\label{2-25}
			\omega_{\ep, \Gamma}=(J')^{-1}\left((\mathcal{G}_s^+\omega_{\ep, \Gamma}- W\ep^{3-2s}x_1-\mu_{\ep, \Gamma})_+\right), \ \ \ \text{in}\ B(\ep^{-1}(d_0,0), \ep^{-1} d_0/2).
		\end{equation}
	\end{corollary}

	 Having made the necessary preparation, we next establish the existence and properties of  maximizers for \eqref{max**}.
	
	\subsubsection{Existence and properties of maximizers}
	We first state some basic properties of maximizers.
	\begin{lemma}\label{lem2-20}
		For each $\ep\in (0,\ep_2)$, there exists a maximizer for \eqref{max**}. Let $\omega_\ep$ be a maximizer of $E_\ep$ over $\mathcal{A}_\ep$. Then the following assertions hold:
		\begin{itemize}
			\item[(i)] $\omega_{\ep }$ is Steiner symmetric with respect to some plane $\{x_2=const.\}$;
			\item[(ii)] There is a constant $\mu_\ep$ such that
			\begin{equation}\label{2-26}
				\omega_{\ep }=(J')^{-1}\left((\mathcal{G}_s^+\omega_{\ep   }- W\ep^{3-2s}x_1-\mu_{\ep })_+\right), \ \ \ \text{in}\ B(\ep^{-1}(d_0,0), \ep^{-1} d_0/2);
			\end{equation}
		\item[(iii)] The energy satisfies $$I_0+O(\ep^{2-2s})\leq E_\ep(\omega_\ep)\leq I_0;$$
			\item[(iv)] There exists a constant $0<C<+\infty$ independent of $\ep$ such that
			\begin{equation}\label{2-27}
				\limsup_{\ep\to0_+}\|\omega_{\ep }\|_\infty\leq C.
			\end{equation}
		\end{itemize}
	\end{lemma}
	\begin{proof}
		Fix a 	$\Gamma> \max\{\Gamma_0, C_1\}$, then $\omega_{\ep, \Gamma}\in \mathcal{A}_{\ep, \Gamma}$ is a maximizer of  $E_\ep$ over $\mathcal{A}_{\ep }$ for $\ep\in (0,\ep_2)$.  That is, we obtain a maximizer of $E_\ep$ over $\mathcal{A}_{\ep }$ for each $\ep$ small, which we denote by $\omega_\ep$ for simplicity.	The properties of these maximizers can be derived by arguments quite similar to those in the preceding subsection, so we omit the details.
	\end{proof}
	
	We note that \eqref{2-26} is not yet sufficient to provide a dynamically possible steady vortex flow for the gSQG equation. This is because of the presence of the truncation function $1_{B(\ep^{-1}(d_0,0), \ep^{-1} d_0/2)}$, which makes $\omega_{\ep }$ and $\mathcal{G}_s^+\omega_{\ep   }- W\ep^{3-2s}x_1-\mu_{\ep }$ may	be not functional dependent in the whole space $\mathbb{R}^2$.   To get the desired solution, we need to prove that the support of $\omega_\varepsilon$ is away from the boundary of $B(\ep^{-1}(d_0,0), \ep^{-1} d_0/2)$. We will show that this is the case when $\varepsilon$ is sufficiently small. It is based on the observation that in order to maximize energy, the diameter of the support of a maximizer can not be too large. We will reach this conclusion in several steps. We begin by giving a lower bound of $\mu_\ep$.

	\begin{lemma}\label{lem2-21}
		For any $\delta>0$, there exist an $\ep_\delta>0$ such that
		\begin{equation}\label{2-28}
			\mu_\ep\geq \mu_*-\delta, \quad \forall \ \ep\in(0,\ep_\delta),
		\end{equation}
	where $\mu_*>0$ is the constant in Lemma \ref{lem2-10}
	\end{lemma}	
	\begin{proof}
		Suppose on the contrary that there are constant $\delta_0>0$ and  a sequence $\{\ep_j\}_{j=1}^\infty$ with $\ep_j\to0$ as $j\to\infty$ such that $\limsup_{j\to+\infty}\mu_{\ep_j}\leq \mu_*-\delta_0.$ By Lemma \ref{lem2-20}, we know that $\{\omega_{\ep_j}\}_{j=1}^\infty$ is a maximizing sequence of $E_0$. Then   Theorem \uppercase\expandafter{\romannumeral2}.2 and Corollary \uppercase\expandafter{\romannumeral2}.1 in \cite{L84}  give  a subsequence (still denoted by $\{\omega_{\ep_j}\}_{j=1}^\infty$ for convenience) converges to a maximizer $\omega_0$ of $E_0$ in $L^1\cap L^{2-s}(\mathbb{R}^2)$ after suitable translations. Since $|x|^{2s-2}\in L^{\frac{1}{1-s},\infty}$ (see \cite{Grf} for the weak $L^p$ spaces), we deduce from the generalized Young's inequality (  Theorem 1.4.25 in \cite{Grf}) that $\mathcal{G}_s \omega_{\ep_j}$ converges to $\mathcal{G}_s \omega_0$ strongly in $L^{\frac{2-s}{(1-s)^2}}$.  Then, extracting another subsequence, one may assume that both $\omega_{\ep_j}$ and $\mathcal{G}_s \omega_{\ep_j}$ converge to $\omega_0$ and $\mathcal{G}_s \omega_0$ a.e. respectively.
		
		On the support of $\omega_{\ep_j}$, the Euler-Lagrange equation \eqref{2-26} implies
		$$\mu_{\ep_j}=\mathcal{G}_s^+\omega_{\ep_j}- W\ep^{3-2s}_jx_1-J'(\omega_{\ep_j})\geq \mathcal{G}_s \omega_{\ep_j} -J'(\omega_{\ep_j})+O(\ep^{2-2s}).$$
		Letting $j\to+\infty$, the a.e. convergence and the assumption on $\mu_{\ep_j}$ imply that
		$$\mu_*-\delta_0\geq  \mathcal{G}_s \omega_0-J'(\omega_0).$$
		on the other hand, by the Euler-Lagrange equation for $\omega_0$ and Lemma \ref{lem2-10}, one has
		$$\mathcal{G}_s \omega_0-J'(\omega_0)=\mu_0,$$ for some $\mu_0\geq \mu_*$, which is a contradiction. The proof of this lemma is thus finished.
	\end{proof}	
	
	Now, we determine the size of the support of $\omega_\ep$.
	\begin{lemma}\label{lem2-22}
		There exists a constant $R>0$ such that for $\ep$ sufficiently small and for any maximizer $\omega_\ep$,   the support of $\omega_\ep$ is contained in a disk of radius $R$.
	\end{lemma}		
	\begin{proof}
		By the previous lemma, we can take $\ep$ small such that $\mu_\ep\geq \frac{\mu_*}{2}$. Then the Euler-Lagrange equation \eqref{2-26} implies
		\begin{equation}\label{2-29}
			J'(\omega_\ep)=(\mathcal{G}_s^+\omega_\ep- W\ep^{3-2s}x_1-\mu_\ep)_+\leq (\mathcal{G}_s \omega_\ep -\frac{\mu_*}{2}+O(\ep^{2-2s}) )_+.
		\end{equation}
		 Theorem \uppercase\expandafter{\romannumeral2}.2 and Corollary \uppercase\expandafter{\romannumeral2}.1 in \cite{L84} provide  a subsequence (still denoted by $\{\omega_\ep\}$), which tends to a maximizer $\omega_0$ of $E_0$ in $L^1\cap L^{2-s}$ after translation. So, for any $\delta>0$, we can  find a constant $\ep_\delta>0$ such that
		\begin{equation*}
			\| \omega_\ep-\omega_{0 }^\ep\|_{2-s}\leq \delta, \ \  \forall \ \ep\in(0,\ep_\delta),
		\end{equation*}
		where $\omega_{0 }^\ep$ is a translation of $\omega_0$.
		By Lemma \ref{lem2-11}, we conclude that $\text{spt}(\omega_{0}^\ep) \subset B(\tilde x_\ep, R_*)$ for some point $\tilde x_\ep$. Thus, we get by using the H\"older inequality
		\begin{equation*}
			\begin{split}
				\int_{B(\tilde x_\ep, R_*)} \omega_\ep&\geq \int_{B(\tilde x_\ep, R_*)} \omega_{0}^{\ep}-\| \omega_\ep-\omega_{0 }^\ep\|_{L^1(B(\tilde x_\ep, R_*))}\\
				&\geq \kappa -C\| \omega_\ep-\omega_{0 }^\ep\|_{2-s}\geq \kappa-C\delta,
			\end{split}
		\end{equation*}
		which, combined with $\int \omega_\ep=\kappa$, implies
		\begin{equation}\label{2-30}
			\int_{B(\tilde x_\ep, R_*)^c}  \omega_\ep\leq C \delta,\ \  \forall \ \ep\in(0,\ep_\delta).
		\end{equation}
		
		Take a large $R\geq R_*$ such that $\frac{c_s \kappa}{|R-R_*|^{2-2s}}\leq \frac{\mu_*}{6}$. Then for any $x\in B(\ep^{-1}(d_0,0), \ep^{-1} d_0/2)\setminus B(\tilde x_\ep, R)$, using Lemma \ref{lem2-1} and \eqref{2-27}, we have
		\begin{align}\label{2-31}
				&\mathcal{G}_s \omega_\ep(x)=\int_{ B(\tilde x_\ep, R_*)} \frac{c_s \omega_\ep(y)}{|x-y|^{2-s}} dy +\int_{ B(\tilde x_\ep, R_*)^c} \frac{c_s \omega_\ep(y)}{|x-y|^{2-s}} dy\nonumber\\
				&\leq \frac{c_s \kappa}{|R-R_*|^{2-2s}}+C\left(\left(\int_{  B(\tilde x_\ep, R_*)^c}  \omega_\ep\right)^a+\left(\int_{ B(\tilde x_\ep, R_*)^c}  \omega_\ep\right)^b\right)\\
				&\leq \frac{\mu_*}{6}+C(\delta^a+\delta^b)\leq \frac{\mu_*}{3},\nonumber
		\end{align}
		by taking $\delta$ small.   \eqref{2-29}  implies that $\omega_\ep(x)=0$ for arbitrary $x\in B(\ep^{-1}(d_0,0), \ep^{-1} d_0/2)\setminus B(\tilde x_\ep, R)$. Hence, we find the constant $R>0$ such that the support of $\omega_\ep$ is contained in a disk of radius $R$ for $\ep$ small and finish the proof.
	\end{proof}	
	
	To find the location of $\omega_\ep$, let $x_\ep:=\kappa^{-1}\int x\omega_\ep$ be the center of mass. Then $x_\ep=	(\ep^{-1} d_\ep,0)$ for some $d_\ep>0$. By replacing $R$ with $2R$, we may assume that for $\ep$ small, $$\text{spt}(\omega_\ep)\subset B(x_\ep, R).$$
	
	\begin{lemma}\label{lem2-23} There holds
		\begin{equation}\label{2-32}
			d_\ep\to d_0, \  \ \text{as}\ \ep\to 0.
		\end{equation}
	\end{lemma}		
	\begin{proof}
		Since $x_\ep\in B(\ep^{-1}(d_0,0), \ep^{-1} d_0/2)$, we deduce $\frac{d_0}{2}< d_\ep<\frac{3d_0}{2}.$ Up to a subsequence, we may assume that $$d_\ep\to d_*\in\bigg[\frac{d_0}{2},  \frac{3d_0}{2}\bigg], \ \ \text{as}\  \ep\to0.$$
		We prove that $d_*=d_0$. Indeed, take $$\bar \omega_\ep(x):=\omega_\ep(x+x_\ep-(\ep^{-1}d_0,0))\in \mathcal{A}_\ep.$$
		Noticing that $\omega_\ep$ is a maximizer, i.e. $E_\ep(\omega_\ep)\geq E_\ep(\bar \omega_\ep)$, we find
		\begin{equation*}
			\begin{split}
				\frac{c_s}{2}\int_{\mathbb{R}^2_+} \frac{\omega_\ep(x)\omega_\ep(y)}{|x-\bar y|^{2-2s}}dxdy& + W\ep^{3-2s}\int_{\mathbb{R}^2_+} x_1\omega_\ep(x) dx\\
				&\leq \frac{c_s}{2}\int_{\mathbb{R}^2_+} \frac{\bar\omega_\ep(x)\bar\omega_\ep(y)}{|x-\bar y|^{2-2s}}dxdy +W\ep^{3-2s}\int_{\mathbb{R}^2_+} x_1\bar\omega_\ep(x) dx.
			\end{split}
		\end{equation*}
		Letting $\ep\to0$ in the above inequality, we obtain
		$$\frac{c_s\kappa}{2^{3-2s}d_*^{2-2s}}+ Wd_*\leq \frac{c_s\kappa}{2^{3-2s}d_0^{2-2s}}+W d_0,$$
		which implies $d_*=d_0$ since $d_0$ is the unique minimizer of the function $h(\tau)=\frac{c_s\kappa}{2^{3-2s}\tau^{2-2s}}+ W\tau$ on $(0,+\infty)$.
		
		 Noting that, by the above proof,  one can see that each sequence  in $\{d_\ep\}$ has a convergent  subsequence that   converges   to the same limit   $d_0$. Then a simple contradiction argument shows that  $\{d_\ep\}$ itself  y converges  to $d_0$. The proof of this lemma is thus complete.
	\end{proof}

	Now we are ready to prove Theorem \ref{thmax*}.
	
	\noindent{\bf Proof of Theorem \ref{thmax*}:} Suppose  $J(t)=\int_0^t f^{-1}(\tau) d\tau$ for some $f$ satisfying $(\mathbf{H_{1}})$ and $(\mathbf{H_{2}})$.  Combining Lemma \ref{lem2-20}--\ref{lem2-23}, we can finish
	proof of Theorem \ref{thmax*}.   \qed
	
	  In view of Lemmas \ref{lem2-22} and \ref{lem2-23}, for sufficiently small $\ep$, the support of $\omega_\ep$ is far away from the boundary of $B(\ep^{-1}(d_0,0), \ep^{-1} d_0/2)$. We extend $\omega_\ep$ to the half-space $\mathbb{R}^2_+$ by defining $\omega_\ep$ to be $0$ in  $\mathbb{R}_+^2\setminus B(\ep^{-1}(d_0,0), \ep^{-1} d_0/2)$. With this fact in hand, we can now show that $\omega_{tr,\ep}(x):=\omega_\ep(x)-\omega_\ep(\bar x)$ is a steady solution in the sense of \eqref{1-3}. More precisely, we have
	\begin{lemma}\label{lem2-24}
		Suppose that $J(t)=\int_0^t f^{-1}(\tau) d\tau$ for some $f$ satisfying $(\mathbf{H_{1}})$ and $(\mathbf{H_{2}})$ with $f\in C^{1-2s}$ if $0<s<\frac{1}{2}$. Let $\omega_\ep$ be a maximizer obtained above and $\omega_{tr,\ep}(x):=\omega_\ep(x)-\omega_\ep(\bar x)$, then provided that $\varepsilon$ is sufficiently small, it holds
		\begin{equation}\label{2-33}
			\int_{\mathbb{R}^2}\omega_{tr,\ep} \nabla^\perp(\mathcal{G}_s\omega_{tr,\ep}-W\ep^{3-2s}x_1)\cdot\nabla \varphi  dx=0, \ \ \ \forall\,\varphi\in C_0^\infty(\mathbb{R}^2).
		\end{equation}
	\end{lemma}
	\begin{proof}
	     Using the  regularity theory for fractional Laplacians (Propositions 2.8 and 2.9 in \cite{Sil}), for $s\not=\frac{1}{2}$, we can prove that   $\mathcal G_s \omega_{tr,\ep} \in C^{0, 1}$  by  our assumptions on $f$, the fact $\omega_\ep\in L^1\cap L^\infty$ and $\mathcal G_s \omega_{tr,\ep} \in L^\infty$ due to Lemma \ref{lem2-1}. For $s=\frac{1}{2}$, we use the standard theory on potentials to deduce that $\mathcal G_s \omega_{tr,\ep} \in W^{2s,p}$ for any $p>1$. Therefore, the integral in \eqref{2-33} makes sense for all $s\in(0,1)$.
	
	     By the definition of  $\omega_{tr,\ep}$, 		it is sufficient  to prove
	\begin{equation*}
		\int_{\mathbb{R}^2_+}\omega_{\ep} \nabla^\perp(\mathcal{G}_s^+\omega_{\ep}-W\ep^{3-2s}x_1)\cdot\nabla \varphi  dx=0, \ \ \ \forall\,\varphi\in C_0^\infty(\mathbb{R}^2_+).
	\end{equation*}
		Recall that
		\begin{equation*}
			\omega_{\ep }=(J')^{-1}\left((\mathcal{G}_s^+\omega_{\ep   }- W\ep^{3-2s}x_1-\mu_{\ep })_+\right),\ \ \forall\,x\in \mathbb{R}_+^2.
		\end{equation*}
		Let $F(t)=\int_0^t (J')^{-1} (\tau)d \tau$. For any $\varphi\in C_0^\infty(\mathbb{R}^2_+)$, we apply integrate by parts to obtain
		\begin{equation*}
			\begin{split}
				&\int_{\mathbb{R}^2_+}\omega_\ep \nabla^\perp(\mathcal{G}_s^+\omega_\varepsilon-W\ep^{3-2s}x_1)\cdot\nabla \varphi dx\\
				&=-\int_{B(\ep^{-1}(d_0,0), \ep^{-1} d_0/2)} F ((\mathcal{G}_s^+\omega_{\ep   }- W\ep^{3-2s}x_1-\mu_{\ep })_+)(\partial_{x_2}\partial_{x_1}\varphi-\partial_{x_1}\partial_{x_2}\varphi)dx=0,
			\end{split}
		\end{equation*}
		 where we have used the fact $(\mathcal{G}_s^+\omega_{\ep   }-W \ep^{3-2s}x_1-\mu_{\ep })_+=0$  on $\partial B(\ep^{-1}(d_0,0), \ep^{-1} d_0/2)$. This completes the proof.
	\end{proof}

	Now we are ready to prove  Theorem \ref{thmE}.

	\noindent{\bf Proof  of Theorem \ref{thmE}:}
	Let $\omega_{tr,\ep}(x):=\omega_\ep(x)-\omega_\ep(\bar x)$. Then the statements of Theorem \ref{thmE} follow from Theorem \ref{thmax*} and Lemma \ref{lem2-24}. \qed
	
	 \section{Uniqueness of maximizers}\label{sec3}	
	 In this section, we show the uniqueness of maximizers in the case  $J(t)=L t^{1+\frac{1}{p}}$  for some $L>0$ and  $p\in(0, \frac{1}{1-s})$. We first establish some finer estimates of the asymptotic behavior by careful analysis. Then, we obtain the uniqueness by using the non-degeneracy of linearized equations (i.e. Proposition \ref{lem2-13} (iii)). In what follows, we sometimes leave out the domain in the   integral symbol  and abbreviate it to $\int$ when there is no risk of confusion.
	\subsection{Refined estimates on asymptotic behaviors}\label{sec3-1}
	In the cases  $J(t)=L t^{1+\frac{1}{p}}$, we study in detail the asymptotic behaviors of the maximizers.
	\begin{lemma}\label{lem2-25}
		Suppose  that $J(t)=L t^{1+\frac{1}{p}}$  for some $L>0$ and  $p\in(0, \frac{1}{1-s})$. For each $\ep$ small, let $ \omega_\ep$ be a  maximizer  of $E_\ep$ over $\mathcal{A}_\ep$. Define $\tilde \omega_\ep(x):=\omega_\ep(x+x_\ep)$, where $x_\ep=\kappa^{-1}\int x\omega_\ep$ is the center of mass of $\omega_\ep$. Let $\omega_0$ be the unique maximizer of $E_0$ over $\mathcal{A}_0$ with $\int x \omega_0=0$. Then, we have
		$$\lim_{\ep\to0}	\|\tilde \omega_\ep-\omega_0\|_\infty=0.$$
	\end{lemma}
	\begin{proof}
		We first prove the convergence of the energy. On the one hand, since $\tilde \omega_\ep\in \mathcal{A}_0$ and $\omega_0$ is  a maximizer of $E_0$ over $\mathcal{A}_0$, we have
		$$E_0(\omega_0)\geq E_0(\tilde \omega_\ep).$$
		On the other hand, noticing that $\omega_0(\cdot-\ep^{-1}(d_0,0))\in \mathcal{A}_\ep$, we deduce
		$$E_\ep(\omega_\ep)\geq E_\ep(\omega_0(\cdot-\ep^{-1}(d_0,0))).$$
		It is easy to see that $E_\ep(\omega_\ep)=E_0(\tilde \omega_\ep)+O(\ep^{2-2s})$ and $E_\ep(\omega_0(\cdot-\ep^{-1}(d_0,0)))=E_0(\omega_0)+O(\ep^{2-2s})$. So, we conclude
		$$E_0(\omega_0)\geq E_0(\tilde \omega_\ep)\geq E_0(\omega_0)+O(\ep^{2-2s}),$$
		which implies
		\begin{equation}\label{2-34}
			\lim_{\ep\to0}E_0(\tilde \omega_\ep)=E_0(\omega_0).
		\end{equation}
		
		Note that for arbitrary $\ep$ small, one has $\|\tilde \omega_\ep\|_\infty=\|  \omega_\ep\|_\infty\leq C$ for some $C>0$ independent of $\ep$  and $\text{spt}(\tilde \omega_\ep)\subset B(0,R)$ for some $R$ independent of $\ep$ by previous lemmas. Therefore, we may assume that up to a subsequence $\tilde\omega_\ep\to \hat\omega_0$ weakly star in $L^\infty$ for some $\hat\omega_0$. Similar argument as the proof of Lemma \ref{lem2-14} shows that $E_0(\hat\omega_0)=\lim_{\ep\to0}E_0(\tilde \omega_\ep)=E_0(\omega_0)$, so $\hat\omega_0$ is a maximizer of   $E_0$ over $\mathcal{A}_0$. Moreover, by the weakly star convergence, we find
		$$\int x \hat\omega_0=\lim_{\ep\to0}\int x \tilde \omega_\ep=0.$$
		This indicates that $ \hat\omega_0=\omega_0$ by the uniqueness of maximizers of $E_0$ due to Proposition \ref{lem2-13}.
		
		We have proved the weakly star convergence. Now, we show the strong convergence. Since $ \omega_\ep$ is uniformly bounded in $L^1\cap L^\infty$, then by Lemma \ref{lem2-1} and   Proposition 2.9 in \cite{Sil}, we have $\mathcal{G}_s\omega_\ep$ is uniformly bounded   in $C^\alpha$ for some $0<\alpha<1$. Then, we deduce from the representation \eqref{2-26} that $ \omega_\ep$ is uniform bounded in $C^\alpha$. So by the Arela-Ascoli theorem, we may assume that up to a  subsequence, $\{\tilde\omega_\ep\}$ strongly converges in $L^\infty$ to $\omega_0$.
		
		Since each sequence  in $\{\tilde\omega_\ep\}$ has a convergent  subsequence that strongly converges in $L^\infty$ to the same limit   $\omega_0$, a simple contradiction argument shows that  $\{\tilde\omega_\ep\}$ itself strongly converges in $L^\infty$ to $\omega_0$. The proof of this lemma is thus finished.
	\end{proof}
	
	We derive the integral representation for the Lagrange multipliers $\mu_\ep$ in terms of $\omega_\ep$.
	\begin{lemma}\label{lem2-26}
		Suppose  that $J(t)=L t^{1+\frac{1}{p}}$  for some $L>0$ and  $p\in(0, \frac{1}{1-s})$. Let $\omega_\ep$ be a maximizer of $E_\ep$ over $\mathcal{A}_\ep$ satisfying \eqref{2-26} for some $\mu_\ep$. Then we have
		\begin{equation}\label{2-35}
			\mu_\ep \kappa =A_\gamma E_0(\omega_\ep)+ B_\gamma W  \ep^{3-2s}\int_{\mathbb{R}^2_+ } x_1\omega_\ep +  C_\gamma \int_{\mathbb{R}^2_+ }\int_{\mathbb{R}^2_+ } \frac{ c_s  \omega_\ep(x)\omega_\ep(y)}{|x-\bar y|^{2-2s}} dx dy,
		\end{equation}
		where the constants $A_\gamma=\frac{2-\gamma-s\gamma}{ 2-s-\gamma}$, $B_\gamma= \frac{2-s}{s-1}-\frac{2-\gamma-\gamma s}{2(s-1)(2-s-\gamma)} $, $C_\gamma=-\frac{2-\gamma-\gamma s}{2(2-s-\gamma)} $ and $\gamma=1+\frac{1}{p}$.
	\end{lemma}
	\begin{proof}
		We may assume that $\omega_\ep$ is symmetric non-increasing in $x_2$. Then Lemmas \ref{lem2-22} and \ref{lem2-23} shows that the support of $\omega_\ep$ is far away from the boundary $\partial B(\ep^{-1}(d_0,0), \ep^{-1} d_0/2)$. So the function $(\omega_\ep)_t(x):=t^{-2}\omega_\ep(t^{-1}x)$ is supported in $B(\ep^{-1}(d_0,0), \ep^{-1} d_0/2)$ for $t\approx 1$ and hence belongs to $\mathcal{A}_\ep$. Since $\omega_\ep$ is a maximizer, we deduce
		\begin{equation*}
			\frac{d E_\ep((\omega_\ep)_t)}{d t} \Bigg{|}_{t=1} =0,
		\end{equation*}
		which, by the definition of $E_\ep$ and straightforward computations, gives	
		$$(s-1)\int_{\mathbb{R}^2_+ } \omega_\ep\mathcal{G}_s^+\omega_\ep- W\ep^{3-2s}\int_{\mathbb{R}^2_+ } x_1\omega_\ep-(2-2\gamma)\int_{\mathbb{R}^2_+ } J(\omega_\ep)=0.$$
		Then, we infer from the above identity that
		\begin{equation}\label{2-36}
			\int_{\mathbb{R}^2_+ } \omega_\ep\mathcal{G}_s \omega_\ep=\frac{2-2\gamma}{s-1}\int_{\mathbb{R}^2_+ } J(\omega_\ep) +\frac{W}{s-1} \ep^{3-2s}\int_{\mathbb{R}^2_+ }  x_1\omega_\ep+c_s\int_{\mathbb{R}^2_+ }\int_{\mathbb{R}^2_+ } \frac{  \omega_\ep(x)\omega_\ep(y)}{|x-\bar y|^{2-2s}} dx dy.
		\end{equation}
		By the definition of $E_0$ and \eqref{2-36}, we deduce
		\begin{equation}\label{2-37}
			\begin{split}
				E_0(\omega_\ep)&=\frac{1}{2}\int_{\mathbb{R}^2_+ } \omega_\ep\mathcal{G}_s \omega_\ep-\int_{\mathbb{R}^2_+ } J(\omega_\ep)\\
				&=\frac{2-s-\gamma}{s-1}\int_{\mathbb{R}^2_+ } J(\omega_\ep)+\frac{W}{2(s-1)} \ep^{3-2s}\int_{\mathbb{R}^2_+ } x_1\omega_\ep+\frac{c_s}{2}\int_{\mathbb{R}^2_+ }\int_{\mathbb{R}^2_+ } \frac{  \omega_\ep(x)\omega_\ep(y)}{|x-\bar y|^{2-2s}} dx dy.
			\end{split}
		\end{equation}
		
		On the other hand,  multiplying the equation $(J') \left(\omega_{\ep   } \right)=(\mathcal{G}_s^+\omega_{\ep   }- W\ep^{3-2s}x_1-\mu_{\ep })_+$ by $\omega_\ep$ and integrating, we have
		\begin{equation}\label{2-38}
			\begin{split}
				\mu_\ep\kappa&=\int_{\mathbb{R}^2_+ } \omega_\ep\mathcal{G}_s^+\omega_\ep- W\ep^{3-2s}\int_{\mathbb{R}^2_+ } x_1\omega_\ep-\gamma \int_{\mathbb{R}^2_+ } J(\omega_\ep)\\
				&=\frac{2-\gamma-\gamma s}{s-1}\int_{\mathbb{R}^2_+ } J(\omega_\ep)+\frac{2-s}{s-1}W \ep^{3-2s}\int_{\mathbb{R}^2_+ } x_1\omega_\ep\\
				&=\frac{2-\gamma-\gamma s}{2-s-\gamma}E_0(\omega_\ep)+\left(\frac{2-s}{s-1}-\frac{2-\gamma-\gamma s}{2(s-1)(2-s-\gamma)} \right) W\ep^{3-2s}\int_{\mathbb{R}^2_+ } x_1\omega_\ep\\
				&\quad -\frac{2-\gamma-\gamma s}{2(2-s-\gamma)} \int_{\mathbb{R}^2_+ }\int_{\mathbb{R}^2_+ } \frac{ c_s \omega_\ep(x)\omega_\ep(y)}{|x-\bar y|^{2-2s}} dx dy,
			\end{split}
		\end{equation}
		where we have used \eqref{2-36} and \eqref{2-37} in   above calculations. So we obtain \eqref{2-35} and finish the proof of this lemma.
	\end{proof}
	
	As an immediate consequence of Lemma \ref{lem2-26}, we obtain the following estimate for $\mu_\ep$.
	\begin{corollary}\label{lem2-27} One has
		$$|\mu_\ep-\mu_0|\to0,\ \ \text{as}\ \ep\to0_+.$$
	\end{corollary}
	
	Next, we derive another integral identity for general $J$, which determines the location of $\omega_\ep$.
	\begin{lemma}\label{lem2-28}
	Suppose that $J$ satisfies $(\mathbf{H_{1}'})$ and $(\mathbf{H_{2}'})$. Let $\omega_\ep$ be a maximizer of $E_\ep$ over $\mathcal{A}_\ep$. Then the following identity is true.
		\begin{equation}\label{2-39}
			2(1-s)c_s  \int_{\mathbb{R}^2_+ }\int_{\mathbb{R}^2_+ } \frac{  \omega_\ep(x)(x_1+y_1)\omega_\ep(y)}{|x-\bar y|^{4-2s}} dx dy= W\ep^{3-2s}\kappa.
		\end{equation}
	\end{lemma}
	\begin{proof}
		We may assume that $\omega_\ep$ is symmetric non-increasing in $x_2$. Then Lemmas \ref{lem2-22} and \ref{lem2-23} show  that the support of $\omega_\ep$ is far away from the boundary $\partial B(\ep^{-1}(d_0,0), \ep^{-1} d_0/2)$. So the function $(\omega_\ep)^t(x):= \omega_\ep( x+(t,0))$ is supported in $B(\ep^{-1}(d_0,0), \ep^{-1} d_0/2)$ for $t\approx 0$ and hence belongs to $\mathcal{A}_\ep$. Since $\omega_\ep$ is a maximizer, we deduce
		\begin{equation*}
			\frac{d E_\ep((\omega_\ep)^t)}{d t} \Bigg{|}_{t=0} =0,
		\end{equation*}
		which, by the definition of $E_\ep$ and straightforward computations, implies \eqref{2-39}. So the proof this lemma is completed.
	\end{proof}
	
	Using the identity \eqref{2-39}, we can sharpen the estimate of $x_\ep$ in Lemma \ref{lem2-23} as follows.
	\begin{lemma}\label{lem2-29}
		Suppose that $J$ satisfies $(\mathbf{H_{1}'})$ and $(\mathbf{H_{2}'})$. Let $\omega_\ep$ be a maximizer of $E_\ep$ over $\mathcal{A}_\ep$. If $x_\ep=\kappa^{-1}\int x\omega_\ep=(\ep^{-1}d_\ep,0)$, then we have
		$$|d_\ep-d_0|=O(\ep^2).$$
	\end{lemma}
	\begin{proof}
		By Lemma \ref{lem2-22}, we know $\text{spt}(\omega_\ep)\subset B(x_\ep, R)$ for some $R$ independent of $\ep$ and $x_\ep=(\ep^{-1}d_\ep,0)$ with $d_\ep\to d_0$ as $\ep\to0_+$.
		We use the Taylor expansion to calculate the left-hand side of \eqref{2-39}
		\begin{align}\label{2-40}
			&\int\int \frac{  \omega_\ep(x)(x_1+y_1)\omega_\ep(y)}{|x-\bar y|^{4-2s}} dx dy= \ep^{3-2s}\int\int \frac{  \tilde \omega_\ep(x)(2d_\ep +\ep(x_1+y_1))\tilde \omega_\ep(y)}{|(2d_\ep ,0)+\ep(x-\bar y)|^{4-2s}} dx dy\nonumber\\
			=&\ep^{3-2s}\int\int \tilde \omega_\ep(x)\tilde \omega_\ep(y)\left[\frac{1}{(2d_\ep)^{3-2s}}+\frac{(2s-3)\ep(x_1+y_1)}{(2d_\ep)^{4-2s}}+O(\ep^2)\right] dx dy\nonumber\\
			=&\frac{\ep^{3-2s}\kappa^2}{(2d_\ep)^{3-2s}} +O(\ep^{5-2s}).
		\end{align}
		Here we have used that     $\int x\tilde \omega_\ep(x)dx=0$, which follows from the definition of $\tilde \omega_\ep$.

		Then using \eqref{2-39} and \eqref{2-40}, we obtain
		$$\frac{	2(1-s)c_s  \kappa^2}{(2d_\ep)^{3-2s}}-W \kappa=O( \ep^2),$$
		which implies
		\begin{equation}\label{2-41}
			d_\ep^{2s-3}=\frac{ 2^{2-2s}W }{(1-s)c_s\kappa}+O( \ep^2)=d_0^{2s-3}+O( \ep^2).
		\end{equation}
		Here we have used the definition of $d_0$.
		Since $\frac{d (t^{2s-3})}{d t} \Big{|}_{t=d_0}\not=0$, we infer from \eqref{2-41} that $|d_\ep-d_0|=O( \ep^2)$ and complete the proof.
	\end{proof}

    \subsection{Proof of the uniqueness} In this subsection, we prove the uniqueness of maximizers when $J(t)=L t^{1+\frac{1}{p}}$  for some $L>0$ and  $p\in(1, \frac{1}{1-s})$. Let $\omega_\ep$ be a maximizer of $E_\ep$ over $\mathcal{A}_\ep$.
	 Recall that $x_\ep=\kappa^{-1}\int x\omega_\ep=(\ep^{-1}d_\ep,0)$ and $\tilde \omega_\ep=\omega_\ep(\cdot+\ep^{-1}(d_\ep,0))$. Set $L_\gamma=\left(\frac{1}{L\gamma}\right)^{\frac{1}{\gamma-1}}$ with $\gamma=1+\frac{1}{p}$. Then, by \eqref{2-26}, \eqref{2-35} and \eqref{2-39}, we obtain the equations of $\tilde \omega_\ep$ and $d_\ep$:
	 \begin{equation}\label{3-1}
	 	\begin{cases}
	 		\tilde \omega_\ep(x)=L_\gamma \left(\mathcal{G}_s\tilde \omega_\ep(x)-\kappa^{-1}A_\gamma E_0(\tilde \omega_\ep)+S_\ep(\tilde\omega_\ep, d_\ep) \right)_+^p,\\
	 		2(1-s)c_s  \int\int \frac{\tilde  \omega_\ep(x)(2d_\ep+\ep(x_1+y_1))\tilde\omega_\ep(y)}{|(2d_\ep,0)+\ep(x-\bar y)|^{4-2s}} dx dy- W \kappa=0,
	 	\end{cases}
	 \end{equation}
	 where the operator 	 	$S_\ep(\tilde \omega_\ep, d_\ep)$ is given by
	 \begin{equation}\label{3-2}
	 	\begin{split}
	 		S_\ep(\tilde \omega_\ep, d_\ep)=&- W\ep^{3-2s}(x_1+\ep^{-1}d_\ep)-\int \frac{\ep^{2-2s} c_s \tilde\omega_\ep(y)}{|(2 d_\ep,0)+\ep(x-\bar y) |^{2-2s}} dy\\
	 		-& \kappa^{-1}B_\gamma W \ep^{2-2s}  d_\ep-\kappa^{-1}C_\gamma \int\int\frac{\ep^{2-2s}c_s\tilde\omega_\ep(x)\tilde\omega_\ep(y)}{|(2d_\ep,0)+\ep(x-\bar y)|^{2-2s}} dxdy.
	 	\end{split}
	 \end{equation}
	 For simplicity of notations, we denote the operators $P_\ep$ and $Q_\ep$ as follows. $$P_\ep(\omega, d):=L_\gamma \left(\mathcal{G}_s  \omega(x)-\kappa^{-1}A_\gamma E_0( \omega )+S_\ep( \omega, d) \right)_+^p,$$
	 and
	 $$Q_\ep(\omega, d):=2(1-s)c_s  \int\int \frac{   \omega(x)(2d +\ep(x_1+y_1)) \omega(y)}{|(2d ,0)+\ep(x-\bar y)|^{4-2s}} dx dy-  W\kappa.$$
	
	 Noting  that for $d_\ep\in (\frac{d_0}{2}, \frac{3d_0}{2})$, $x,y\in B(0,R)$ and $\ep$ small, we always have $|(2 d_\ep,0)+\ep(x-\bar y) |$ is bounded from below by a positive constant. Therefore, we easily obtain the following lemma, whose proof we leave to readers.
	 \begin{lemma}\label{lem3-1}
	 	For $(\omega, d)\in \mathcal{X}^1:=L^1(B(0,R))\times (\frac{d_0}{2}, \frac{3d_0}{2})$ and $\ep$ sufficiently small, there holds
	 	$$\|S_\ep(\omega,d )\|_\infty +\|\nabla S_\ep(\omega,d )\|_{\mathcal{X}^1\to L^\infty}+\|\nabla^2 S_\ep(\omega,d )\|_{\mathcal{X}^1\times \mathcal{X}^1\to L^\infty}=O(\ep^{2-2s}).$$
	 \end{lemma}

	 Using the definition of $E_0$ and \eqref{2-7}, one can see that
	 $$E'_0(\omega_0)\phi=\mu_0\int \phi.$$
	 If we define $$P_0(\omega, d):=L_\gamma \left(\mathcal{G}_s  \omega(x)-\kappa^{-1}A_\gamma E_0( \omega ) \right)_+^p,$$
	 and
	 $$Q_0(\omega, d):=\frac{2(1-s)c_s }{(2d)^{3-2s}} \int\int    \omega(x)  \omega(y) dx dy-  W\kappa.$$
	 By direct calculations, one can verify that $$\omega_0=P_0(\omega_0, d_0),\ \ \nabla P_0(\omega_0, d_0)(\phi,l)= pL_\gamma (\mathcal{G}_s \omega_0-\kappa^{-1}A_\gamma E_0(\omega_0))_+^{p-1}\left(\mathcal{G}_s \phi-\kappa^{-1}A_\gamma \mu_0\int \phi \right)$$ and $$Q_0(\omega_0,d_0)=0, \ \   \nabla Q_0(\omega_0,d_0)(\phi, l)=\frac{ (1-s)c_s \kappa}{2^{1-2s}d_0^{3-2s}}\int \phi-\frac{ (1-s)(3-2s)c_s\kappa^2}{2^{2-2s}d_0^{4-2s}} l.$$

	  By Lemmas \ref{lem2-23} and \ref{lem2-25}, $\tilde \omega_\ep$ and $d_\ep$ take the following form:
	 $$\tilde \omega_\ep=\omega_0+\phi_\ep,\ \  d_\ep=d_0+l_\ep,$$
	 where $\phi_\ep\in L^\infty(B(0,R))$ satisfies $\int x\phi_\ep=\int\phi_\ep=0$,  $\|\phi_\ep\|_\infty=o(1)$ and  $l_\ep$ is a real number with $|l_\ep|=o(1)$ as $\ep\to0$.
	
	 Then we deduce from \eqref{3-1} that $(\phi,l)=(\phi_\ep,l_\ep)$ solves the following system
	 \begin{equation}\label{3-3}
	 	\begin{cases}
	 		\phi-pL_\gamma (\mathcal{G}_s \omega_0-\kappa^{-1}A_\gamma E_0(\omega_0))_+^{p-1}\left(\mathcal{G}_s \phi-\kappa^{-1}A_\gamma\mu_0\int \phi \right) =\mathcal{R}_1(\phi,l),\\
	 		\frac{ (1-s)c_s \kappa}{2^{1-2s}d_0^{3-2s}}\int \phi-\frac{ (1-s)(3-2s)c_s\kappa^2}{2^{2-2s}d_0^{4-2s}} l=\mathcal{R}_2(\phi,l),
	 	\end{cases}
	 \end{equation}
	 where the super-linear terms $\mathcal{R}_1(\phi,l)$ and $\mathcal{R}_2(\phi,l)$ are given by
	 \begin{equation}\label{3-4}
	 	\begin{split}
	 		&\mathcal{R}_1(\phi,l)=L_\gamma \left(\mathcal{G}_s(\omega_0+\phi)-\kappa^{-1}A_\gamma E_0(\omega_0+\phi)+S_\ep(\omega_0+\phi, d_0+l) \right)_+^p\\
	 		-&L_\gamma \left(\mathcal{G}_s \omega_0 -\kappa^{-1}A_\gamma E_0(\omega_0 )  \right)_+^p-pL_\gamma (\mathcal{G}_s \omega_0-\kappa^{-1}A_\gamma E_0(\omega_0))_+^{p-1}\left(\mathcal{G}_s \phi-\kappa^{-1}A_\gamma\mu_0\int \phi \right) \end{split}
	 \end{equation}
	 and
	 \begin{equation}\label{3-5}
	 	\begin{split}
	 		\mathcal{R}_2(\phi,l)&=	2(1-s)c_s  \int\int \frac{(\omega_0+\phi)(x)(2d_0+2l_\ep+\ep(x_1+y_1))(\omega_0+\phi)(y)}{|(2d_0+2l_\ep,0)+\ep(x-\bar y)|^{4-2s}} dx dy- W\kappa\\
	 		-&\frac{ (1-s)c_s \kappa}{2^{1-2s}d_0^{3-2s}}\int \phi+\frac{ (1-s)(3-2s)c_s\kappa^2}{2^{2-2s}d_0^{4-2s}} l \end{split}
	 \end{equation}
	
	 To study the linearized equation \eqref{3-3}, we set $$\mathcal{L}_0\phi:=\phi-pL_\gamma (\mathcal{G}_s \omega_0-\kappa^{-1}A_\gamma E_0(\omega_0))_+^{p-1}\left(\mathcal{G}_s \phi-\kappa^{-1}A_\gamma\mu_0\int \phi \right).$$
	 Using the non-degeneracy Proposition \ref{lem2-13} (iii), we have the following key estimate of $\mathcal{L}_0$.
	 \begin{lemma}\label{lem3-2}
	 	There exists a constant $C_0>0$ such that
	 	$$\|\mathcal{L}_0\phi\|_2\geq C_0\|\phi\|_2, \quad \forall\ \phi \in L^2(B(0,R)) \ \text{with }\ \int x\phi =\int\phi=0.$$
	 \end{lemma}
	 \begin{proof}
	 	We normalize $\phi$ by replacing $\phi$ with $\phi/\|\phi\|_2$ so that $\|\phi\|_2=1$. Now suppose on the contrary that there is  a sequence $\{\phi_n\}_{n=1}^\infty\subset L^2(B(0,R))$ such that
	 	$$\|\phi_n\|=1, \ \int x\phi_n =\int\phi_n=0,$$ while $$\|\mathcal{L}_0\phi_n\|\to0,\ \ \text{as}\ n\to+\infty.$$
	 	That is
	 	\begin{equation}\label{3-6}
	 		\phi_n=pL_\gamma (\mathcal{G}_s \omega_0-\kappa^{-1}A_\gamma E_0(\omega_0))_+^{p-1} \mathcal{G}_s \phi_n+o(1).
	 	\end{equation}	
	 	We first assume that up to a subsequence $\phi_n\to \phi_\infty$ weakly in $L^2$. Note that the support of $(\mathcal{G}_s \omega_0-\kappa^{-1}A_\gamma E_0(\omega_0))_+^{p-1}$ is contained in $B(0,R)$. Then by the regularity theory on potentials and the compact embedding theorem for fractional Sobolev spaces (see e.g. \cite{DiN, S70}), we have up to a subsequence $pL_\gamma (\mathcal{G}_s \omega_0-\kappa^{-1}A_\gamma E_0(\omega_0))_+^{p-1} \mathcal{G}_s \phi_n\to pL_\gamma (\mathcal{G}_s \omega_0-\kappa^{-1}A_\gamma E_0(\omega_0))_+^{p-1} \mathcal{G}_s \phi_\infty$ strongly in $L^2(B(0,R))$ and hence by \eqref{3-6}, we conclude  $\phi_n\to \phi_\infty$ strongly in $L^2$. Thus, we have $\phi_\infty$ satisfies
	 	$$\phi_\infty=pL_\gamma (\mathcal{G}_s \omega_0-\kappa^{-1}A_\gamma E_0(\omega_0))_+^{p-1} \mathcal{G}_s \phi_\infty,\ \ \|\phi_\infty\|_2=1, \ \  \int x\phi_\infty =\int\phi_\infty=0.$$
	 	By the regularity theory on potentials, Sobolev embedding and bootstrap argument, one can prove that $\phi_\infty\in L^\infty(B(0,R))$. Thus, we infer from  Proposition \ref{lem2-13} (iii) that $\phi_\infty\in \text{span}\{  \partial_{x_1}\omega_0, \partial_{x_2}\omega_0\}$, which combined with $ \int x\phi_\infty=0$ implies $\phi_\infty\equiv 0$ due to the radial symmetry of $\omega_0$.  This is a contradiction with  $\|\phi_\infty\|_2=1$. Therefore,   this lemma must hold true.
	 \end{proof}

	 Now we study the right-hand side of \eqref{3-3}.
	 \begin{lemma}\label{lem3-3}
	 	For any $\epsilon>0$, there is a $\delta>0$ such that if $\phi_1, \phi_2\in L^\infty(B(0,R))$ and $l_1, l_2\in \mathbb{R}$ satisfy $\int \phi_1=\int \phi_2=0$ and $$\|\phi_1\|_\infty+\|\phi_2\|_\infty+|l_1|+|l_2|\leq \delta,$$ then we have
	 	$$\limsup_{\ep\to0}(\|\mathcal{R}_1(\phi_1,l_1)-\mathcal{R}_1(\phi_2,l_2)\|_2+|\mathcal{R}_2(\phi_1,l_1)-\mathcal{R}_2(\phi_2,l_2)| )\leq \epsilon (\|\phi_1-\phi_2\|_2+|l_1-l_2|).$$
	 \end{lemma}
 \begin{proof}
 	By direct calculations, we find
 	\begin{align*}
 		&\mathcal{R}_1(\phi_1,l_1)-\mathcal{R}_1(\phi_2,l_2)\\
 		=&L_\gamma \left(\mathcal{G}_s(\omega_0+\phi_1)-\kappa^{-1}A_\gamma E_0(\omega_0+\phi_1)+S_\ep(\omega_0+\phi_1, d_0+l_1) \right)_+^p\\
 		&- L_\gamma \left(\mathcal{G}_s(\omega_0+\phi_2)-\kappa^{-1}A_\gamma E_0(\omega_0+\phi_2)+S_\ep(\omega_0+\phi_2, d_0+l_2) \right)_+^p\\
 		&-pL_\gamma (\mathcal{G}_s \omega_0-\kappa^{-1}A_\gamma E_0(\omega_0))_+^{p-1}\mathcal{G}_s \left(\phi_1-\phi_2  \right)\\
 		=&pL_\gamma \int_0^1\left[\left(\mathcal{G}_s(\omega_0+\phi(\tau))-\kappa^{-1}A_\gamma E_0(\omega_0+\phi(\tau))+S_\ep(\omega_0+\phi(\tau), d_0+l(\tau)) \right)_+^{p-1}\right.\\
 		&\left. - (\mathcal{G}_s \omega_0-\kappa^{-1}A_\gamma E_0(\omega_0))_+^{p-1}\right] d\tau  \mathcal{G}_s(\phi_1-\phi_2) +O(\|\nabla S_\ep(\omega_0, d_0)\|)\|\phi_1-\phi_2\|_2,
 	\end{align*}
 	where $\phi(\tau)=\tau\phi_1+(1-\tau)\phi_2,\ l(\tau)=\tau l_1+(1-\tau)l_2$. To continue, we expand $\mathcal{G}_s(\omega_0+\phi(\tau))-\kappa^{-1}A_\gamma E_0(\omega_0+\phi(\tau))+S_\ep(\omega_0+\phi(\tau), d_0+l(\tau))$ as follows.
 	\begin{align*}
 		&\mathcal{G}_s(\omega_0+\phi(\tau))-\kappa^{-1}A_\gamma E_0(\omega_0+\phi(\tau))+S_\ep(\omega_0+\phi(\tau), d_0+l(\tau))\\
 		=&\mathcal{G}_s \omega_0-\kappa^{-1}A_\gamma E_0(\omega_0 )\\
 		&+\mathcal{G}_s\phi(\tau)+\int \phi(\tau)\mathcal{G}_s\omega_0-\int (J(\omega_0+\phi(\tau))-J(\omega_0))+\frac{1}{2}\int \phi(\tau)\mathcal{G}_s\phi(\tau)+O(\|S_\ep\|_\infty)\\
 		=&\mathcal{G}_s \omega_0-\kappa^{-1}A_\gamma E_0(\omega_0 )+O(\|\phi(\tau)\|_\infty+\|\phi(\tau)\|_\infty^2+\|S_\ep\|_\infty),
 	\end{align*}
 	where we have used 	  H\"older's inequality, Lemmas \ref{lem2-1} and \ref{lem2-2}.
 	Using $h(t)=t_+^{p-1}\in C^{0,\alpha_p}$ with $\alpha_p=\min\{1,p-1\}\in (0,1]$, we continue calculating $\mathcal{R}_1(\phi_1,l_1)-\mathcal{R}_1(\phi_2,l_2)$.
 	\begin{align*}
 		&\mathcal{R}_1(\phi_1,l_1)-\mathcal{R}_1(\phi_2,l_2)\\
 		=& O((\|\phi(\tau)\|_\infty+\|\phi(\tau)\|_\infty^2+\|S_\ep\|)^{\alpha_p}) \mathcal{G}_s(\phi_1-\phi_2) +O(\|\nabla S_\ep(\omega_0, d_0)\|)\|\phi_1-\phi_2\|
 	\end{align*}
 	
 	Thus, if $\|\phi_1\|, \|\phi_2\|\leq \delta$, then we obtain
 	\begin{equation}\label{3-7}
 		\|\mathcal{R}_1(\phi_1,l_1)-\mathcal{R}_1(\phi_2,l_2)\|_2\leq C(\delta^{\alpha_p}+\ep^{(2-2s)\alpha_p})\|\phi_1-\phi_2\|_2=o_{\delta,\ep}(1)\|\phi_1-\phi_2\|_2.
 	\end{equation}
 	Since $\mathcal{R}_2$ is $C^2$ smooth, it is easy to see that
 	\begin{equation}\label{3-8}
 		\|\mathcal{R}_2(\phi_1,l_1)-\mathcal{R}_2(\phi_2,l_2)\|\leq C( \|\phi_1-\phi_2\|^2+|l_1-l_2|^2)=o_{\delta, \ep }(1)( \|\phi_1-\phi_2\| +|l_1-l_2| ).
 	\end{equation}
 The proof of this lemma is thus finished.
 \end{proof}
	
Now, we are ready to prove our uniqueness.

\noindent{\bf Proof of Theorem \ref{thmU}:}

Suppose on the contrary that there are two different maximizers $\omega_{1,\ep}\not=\omega_{2,\ep}$. Then, we obtain two pairs $(\phi_{1,\ep}, l_{1,\ep})$ and $(\phi_{2,\ep}, l_{2,\ep})$ such that $\|\phi_{1,\ep}-\phi_{2,\ep}\|_2+|l_{1,\ep}-l_{2,\ep}|\not=0$, $\int x\phi_{1,\ep}=\int\phi_{1,\ep}=\int x\phi_{2,\ep}= \int\phi_{2,\ep}=0$ and $\|\phi_{1,\ep}\|_\infty+\|\phi_{2,\ep}\|_\infty+|l_{1,\ep}|+|l_{2,\ep}|=o(1)$ as $\ep\to0$. More over, both $(\phi_{1,\ep}, l_{1,\ep})$ and $(\phi_{2,\ep}, l_{2,\ep})$  satisfy the system \eqref{3-3}.

On the one hand, by Lemma \ref{lem3-2},  the difference between the left-hand side satisfies
$$\|\mathcal L_0(\phi_{1,\ep}-\phi_{2,\ep})\|_2+ \frac{ (1-s)(3-2s)c_s\kappa^2}{2^{2-2s}d_0^{4-2s}} |l_{1,\ep}-l_{2,\ep}|\geq C_0\|\phi_{1,\ep}-\phi_{2,\ep}\|_2 +\frac{ (1-s)(3-2s)c_s\kappa^2}{2^{2-2s}d_0^{4-2s}} |l_{1,\ep}-l_{2,\ep}|.$$

On the other hand, by Lemma \ref{lem3-3}, the difference between the right-hand side satisfies
$$\|\mathcal{R}_1(\phi_{1,\ep} ,l_{1,\ep})-\mathcal{R}_1(\phi_{2,\ep},l_{2,\ep})\|_2+|\mathcal{R}_2(\phi_{1,\ep} ,l_{1,\ep})-\mathcal{R}_2((\phi_{2,\ep},l_{2,\ep})|=o(1)(\|\phi_{1,\ep}-\phi_{2,\ep}\|_2+|l_{1,\ep}-l_{2,\ep}|),$$
which leads to a contradiction for $\ep$ small. Therefore, we have established the uniqueness of maximizers for $\ep$ small and completed the proof of Theorem \ref{thmU}. \qed

	 \section{Nonlinear orbital stability}\label{sec4}
	 This section is devoted to investigating the nonlinear stability of traveling solutions obtained in Section \ref{sec2}.  We first prove a general stability theorem in a similar spirit as \cite{BNL13}, where the stability of vortex pairs for the 2D Euler equation was considered. Throughout this section, we always assume that $D\subset \mathbb{R}_+^2$ is the domain $D=  \mathbb{R}_+^2$ if $s>\frac{1}{2}$ and $D=\{x_1\geq1\}$ if $0<s\leq \frac{1}{2}$.
	
	 \subsection{A general stability theorem on the set of maximizers}\label{sec4-1}

	 Let $\xi$ be a non-negative Lebesgue integrable function on $\mathbb{R}^2$, we denote by $\mathcal{R}(\xi)$ the set of (equimeasurable) rearrangements of $\xi$  on $D$ defined by
	 \begin{equation*}
	 	\mathcal{R}(\xi)=\Big\{0\leq \zeta\in L^1(D)\Big| |\{x: \zeta(x)>\tau\}| =|\{x: \xi(x)>\tau\}| , \forall\, \tau>0  \Big\}.
	 \end{equation*}
	 Note that all functions in $\mathcal{R}(\xi)$ have the same $L^q$ norm. Following \cite{BNL13}, we also define
	 \begin{equation*}
	 	\mathcal{R}_+(\xi)=\Big\{\zeta 1_S \big|	\zeta\in\mathcal{R}(\xi),\  \   S\subset D \ \text{measurable}  \Big\},
	 \end{equation*}
	 and
	 $$\overline{\mathcal{R}(\xi)^w}=\left\{ \zeta\geq 0 \  \text{measurable} \bigg  | \  \int_{D}(\zeta-\alpha)_+dx \leq \int_{D}(\xi-\alpha)_+dx, \ \ \forall \alpha>0\right\}.   $$
	 It is easy to see that the inclusions $ 	\mathcal{R}(\xi)\subset \mathcal{R}_+(\xi)\subset \overline{\mathcal{R}(\xi)^w}$ hold. The key fact is that $\overline{\mathcal{R}(\xi)^w}$ is convex and is the weak closure of  $ 	\mathcal{R}(\xi)$ in $L^p$ (see \cite{BNL13, Dou}).
	
	  We denote the   kinetic energy	as
	 \begin{equation*}
	 	E(\zeta)=\frac{1}{2}\int_{D}\zeta(x)\mathcal{G}_s^+\zeta(x) dx,
	 \end{equation*}	
	 and the impulse
	 \begin{equation*}
	 	I(\zeta)= \int_{D} x_1\zeta(x)  dx.
	 \end{equation*}	For a constant $\mathcal W$, set the energy functional as $$\tilde E_{\mathcal W}(\zeta):=\frac{1}{2}\int_{D}\zeta(x)\mathcal{G}_s^+\zeta(x )  dx-\mathcal W \int_{D} x_1\zeta(x) dx.$$
	 For a function $\zeta_0$ and constant $\mathcal W>0$, we will consider the maximization problem
	 $$\sup_{\zeta\in\overline{\mathcal{R}(\zeta_0)^w}}  \tilde E_{\mathcal W}(\zeta).$$
	
	 The following two lemmas are needed.
	 \begin{lemma}\label{lem4-1}
	 	Suppose $\zeta\in L^1\cap L^r(D)$ for some $s^{-1}<r\leq +\infty$ if $0<s\leq \frac{1}{2}$ and $\frac{2}{2s-1}<r\leq +\infty$ if $\frac{1}{2}<s<1$. Then, one has
	 	\begin{equation}\label{4-1}
	 		|\mathcal{G}_s^+ \zeta(x)| \leq C\left( \|\zeta\|_r+   \ \|\zeta\|_1\right)\min\{x_1, x_1^{2s-\frac{2}{r}}\},\ \ \forall\ x\in D.
	 	\end{equation}
	 \end{lemma}
	 \begin{proof}
	 	We first consider the case $x_1\geq 1$. By the mean value theorem, there holds
	 	\begin{equation*}
	 		G_s^+(x,y)\leq \frac{4c_s x_1y_1}{|x-y|^{4-2s}}, \ \ \forall x, y\in  \mathbb{R}_+.
	 	\end{equation*}
	 	Therefore, using H\"older's inequality and  $y_1<x_1+|x-y|\leq 2|x-y|$ if $|x-y|>x_1$, we have
	 	\begin{align*}
	 		|\mathcal{G}_s^+\zeta(x)|&\leq \int_{\mathbb{R}_+^2} G_s^+(x,y)|\zeta(y)| dy\\
	 		&\leq \int_{|x-y|\leq x_1} \frac{c_{2,s}}{|x-y|^{2-2s}} |\zeta(y)| dy + \int_{\{|x-y|> x_1\}} \frac{4x_1y_1}{|x-y|^{4-2s}} |\zeta(y)| dy\\
	 		&\leq C\left(x_1^{2s-\frac{2}{r}}\|\zeta\|_r+ x_1^{2s-2}\|\zeta\|_1\right).
	 	\end{align*}
	 	This proves \eqref{4-1} in the case $0<s< 1$ and $x_1>1$. Now, we turn to the remaining  case $\frac{1}{2}<s<1$ and $0<x_1<1$. By H\"older's inequality, we find
	 	\begin{align*}
	 		|\nabla\mathcal{G}_s^+\zeta(x)|&\leq C\int_{\mathbb{R}_+}  \frac{|\zeta(y)|}{|x-y|^{3-2s}} dy\\
	 		&\leq \int_{|x-y|\leq  1} \frac{c_{2,s}}{|x-y|^{3-2s}} |\zeta(y)| dy + \int_{\{|x-y|>  1\}}  |\zeta(y)| dy\\
	 		&\leq C\left( \|\zeta\|_r+  \|\zeta\|_1\right).
	 	\end{align*}
	 	Noticing that $\mathcal{G}_s^+\zeta(x)\Big|_{x_1=0}\equiv0$, we conclude $$|\mathcal{G}_s^+\zeta(x)|\leq x_1\|\nabla\mathcal{G}_s^+\zeta\|_\infty \leq C\left( \|\zeta\|_q+  \|\zeta\|_1\right)x_1.$$
	 	The proof is thus finished.
	 \end{proof}

	 \begin{lemma}\label{lem4-2}
	 	Suppose $ x_1\zeta\in L^1(\mathbb{R}_+^2)$ and $\zeta\in L^1\cap L^r(D)$ for some $r$ with $s^{-1}<r<+\infty$ if $0<s\leq \frac{1}{2}$ and $\frac{2}{2s-1}<r< +\infty$ if $\frac{1}{2}<s<1$. If $\zeta$ is Steiner symmetric in the $x_2$-variable, then for $x\in D$, there holds
	 	\begin{equation}\label{4-2}
	 		|\mathcal{G}_s^+ \zeta(x)| \leq C\left(\left(|x_2|^{-\frac{1}{2r}} +   |x_2|^{-\frac{1}{2 }} \right)\left( \|\zeta\|_r+   \ \|\zeta\|_1\right)\min\{1, x_1^{2s-\frac{2}{r}-1}\}  +|x_2|^{s-2}\|x_1\zeta\|_1\right)x_1.
	 	\end{equation}
	 \end{lemma}
	 \begin{proof}
	 	For $x\in\mathbb{R}_+^2$ fixed, let
	 	\begin{equation*}
	 		\zeta_1(y)=\left\{
	 		\begin{array}{lll}
	 			\zeta(y), \ \  & \text{if} \ \ |y_2-x_2|<\sqrt{|x_2|},\\
	 			0, & \text{if} \ \ |y_2-x_2|\geq\sqrt{|x_2|}.
	 		\end{array}
	 		\right.
	 	\end{equation*}
	 	Using  equation (2.11) in \cite{Bu0} (see also (6) in \cite{Bu6}), it is easy to see that for any $1\leq q\leq r$
	 	\begin{equation*}
	 		\|\zeta_1\|_q\leq \left(\frac{|x_2|^{\frac{1}{2}}}{|x_2|}\right)^{\frac{1}{q}}\|\zeta\|_q=|x_2|^{-\frac{1}{2q}}\|\zeta\|_q.
	 	\end{equation*}
	 	Hence, by \eqref{4-1}, we have
	 	\begin{equation}\label{4-3}
	 		\begin{split}
	 			|\mathcal{G}_s^+\zeta_1(x)|&\leq C\left(\left( \|\zeta_1\|_r+   \ \|\zeta_1\|_1\right)\min\{1, x_1^{2s-\frac{2}{r}-1}\}\right)x_1\\
	 			&\leq C\left(\left(|x_2|^{-\frac{1}{2r}} +   |x_2|^{-\frac{1}{2 }} \right)\left( \|\zeta\|_r+   \ \|\zeta\|_1\right)\min\{1, x_1^{2s-\frac{2}{r}-1}\} \right)x_1.
	 		\end{split}	
	 	\end{equation}
	 	Letting $\zeta_2=\zeta-\zeta_1$, we have
	 	\begin{align}\label{4-4}
	 		|\mathcal{G}_s^+\zeta_2(x)|&= c_{s}\int_{|x_2-y_2|>\sqrt{|x_2|}} \left(\frac{1}{|x-y|^{2-2s}}-\frac{1}{|x-\bar{y}|^{2-2s}}\right)|\zeta(y)|dy\nonumber\\
	 		& \leq C\int_{|x-y|>\sqrt{|x_2|}} \frac{x_1y_1}{|x-y|^{4-2s}}\zeta(y)dy \\
	 		&\leq \frac{Cx_1}{|x_2|^{2-s}}\|x_1\zeta\|_1,\nonumber
	 	\end{align}	
	 	which, together with \eqref{4-3}, gives \eqref{4-2} and completes the proof.
	 \end{proof}

	 The following property enables us to control the supports of maximizers.
	 \begin{lemma}\label{lem4-3}
	 	Suppose that $\zeta \in L^1\cap L^q(D)$ with some $s^{-1}<q<\infty$   and $\mathcal W>0$ is a given constant.
	 	Let $h=\zeta 1_V$ for  some set $V\subset \{\mathcal{G}_s^+\zeta- \mathcal W  x_1 \leq 0\}$, then $$\tilde E_{\mathcal W}(\zeta-h)\geq \tilde E_{\mathcal W}(\zeta)$$ with strict inequality unless $h\equiv0$.
	 \end{lemma}	
	 \begin{proof}
	 	It is easy to see that
	 	\begin{equation*}
	 		\begin{split}
	 			\tilde E_{\mathcal W}(\zeta-h)&=\frac{1}{2}\int_{\mathbb R^2_+}(\zeta-h)(x)\mathcal{G}_s^+(\zeta-h)(x) d\nu- \mathcal W \int_{\mathbb R^2_+}(\zeta-h)(x) x_1  dx\\
	 			&=\tilde E_{\mathcal W}(\zeta)+\frac{1}{2}\int_{\mathbb R^2_+} h\mathcal{G}_s^+h -\int_{\mathbb R^2_+}h(x)\left(\mathcal{G}_s^+\zeta(x) - \mathcal W  x_1 \right)dx\\
	 			&\geq \tilde E_{\mathcal W}(\zeta)+\frac{1}{2}\int_{\mathbb R^2_+} h\mathcal{G}_s^+h,
	 		\end{split}
	 	\end{equation*}
	 	which implies $\tilde E_{\mathcal W}(\zeta-h)\geq \tilde E_{\mathcal W}(\zeta)$ since $\frac{1}{2}\int h\mathcal{G}_s^+h\geq 0$ and $\frac{1}{2}\int h\mathcal{G}_s^+h=0$ if and only if $h\equiv0$. The proof is thus complete.
	 \end{proof}

	 \begin{lemma}\label{lem4-4}
	 	Let  $0\leq \zeta_0\in L^1\cap L^q(D)$ with some  $q$ with $s^{-1}<q\leq +\infty$ if $0<s\leq \frac{1}{2}$ and $\frac{2}{2s-1}<q\leq +\infty$ if $\frac{1}{2}<s<1$  and $\mathcal W>0$ is a given constant. Then $$\sup_{\zeta\in\overline{\mathcal{R}(\zeta_0)^w}}  \tilde E_{\mathcal W}(\zeta)<+\infty,$$ and any maximizer (if exists) is supported in $[0, M_0]\times \mathbb R$, where $M_0$ is a  constant depending on $\|\zeta_0\|_1+\|\zeta_0\|_q$ and $\mathcal W$.
	 \end{lemma}
	 \begin{proof}
	 	The upper bounded of $\tilde E_{\mathcal W}$ over $ \overline{\mathcal{R}(\zeta_0)^w}$ follows from Lemma \ref{lem2-2}. By Lemma \ref{lem4-1} and the fact that $\zeta_0\in L^r$ for any $r\in[1,q]$, there is a constant $M_0$ depending on $\|\zeta_0\|_1+\|\zeta_0\|_q$ and $\mathcal W$ such that $\mathcal{G}_s^+\zeta(x)-\mathcal W x_1\leq 0$ for all $x\in D$ with $x_1\geq M_0$ and for any $\zeta\in  \overline{\mathcal{R}(\zeta_0)^w}$.  Suppose that $\zeta^0 \in \overline{\mathcal{R}(\zeta_0)^w}$ is a maximizer, let $h:=\zeta^0 1_{(M_0, \infty)\times \mathbb R}$, then we infer from Lemma \ref{lem4-3}   that $h\equiv0$ since $(M_0, \infty)\times \mathbb R\subset\{\mathcal{G}_s^+\zeta^0- \mathcal W  x_1 \leq 0\}$. The proof of this lemma is thus finished.
	 \end{proof}

	 To obtain the compactness of maximizing sequences, we need the following concentration compactness lemma  due to Lions \cite{L84}.
	 \begin{lemma}\label{lem4-6}
	 	Let $\{u_n\}_{n=1}^\infty$ be a sequence of nonnegative functions in $L^1(D)$ satisfying
	 	$$\limsup_{n\rightarrow \infty} \int_{D} u_n dx\rightarrow \mu,$$ for some $0\leq \mu<\infty$.
	 	Then, after passing to a subsequence, one of the following holds:\\
	 	(i) (Compactness) There exists a sequence $\{y_n\}_{n=1}^\infty$ in $\overline{\mathbb{R}_+^2}$ such that for arbitrary $\varepsilon>0$, there exists $R>0$ satisfying
	 	\begin{equation*}
	 		\int_{D\cap B_R(y_n)}u_n dx\geq \mu-\varepsilon, \quad \forall n\geq 1.
	 	\end{equation*}\\
	 	(ii) (Vanishing) For each $R>0$,
	 	\begin{equation*}
	 		\lim_{n\rightarrow \infty}\sup_{y\in D}  \int_{D\cap B_R(y_n)} u_n dx =0.
	 	\end{equation*} \\
	 	(iii) (Dichotomy) There exists a constant $0<\alpha<\mu$ such that for any $\varepsilon>0$, there exist $N=N(\varepsilon)\geq 1$ and $0\leq u_{i,n}\leq u_n, \,i=1,2$ satisfying
	 	\begin{equation*}
	 		\begin{cases}
	 			\|u_n-u_{1,n}-u_{2,n}\|_{L^1(D)}+|\alpha-\int_{D} u_{1,n} dx|+|\mu-\alpha-\int_{D} u_{2,n} dx|<\varepsilon,\quad \text{for}\,\,n\geq N,\\
	 			d_n:=\text{dist}(\text{spt}(u_{1,n}), \text{spt}(u_{2,n}))\rightarrow \infty, \quad \text{as}\,\,n\rightarrow \infty.
	 		\end{cases}	
	 	\end{equation*}
	 	Moreover, if $\mu=0$ then  only vanishing will occur.
	 \end{lemma}
	 \begin{proof}
	 	This lemma is a slight reformulation of Lemma 1.1 in \cite{L84}, so we omit the proof.
	 \end{proof}
	
	 For a function $0\leq \zeta_0\in L^1\cap L^q(D)$  and a constant $\mathcal{W}>0$, we denote $S_{\zeta_0,\mathcal{W}}:=\sup_{\zeta\in\overline{\mathcal{R}(\zeta_0)^w}}  \tilde E_{\mathcal W}(\zeta)$ as the maximum value and $\tilde \Sigma_{\zeta_0,\mathcal{W}}:=\{\zeta\in\overline{\mathcal{R}(\zeta_0)^w}\mid \tilde E_{\mathcal W}(\zeta)=  S_{\zeta_0,\mathcal{W}}\}$ as the set of all the maximizers. To continue,  we  first show the   compactness  of    maximizing sequences  by using   Lemma \ref{lem4-6}.
	 \begin{proposition}\label{prop4-7}
	 	For $q$ with $\max\{2, s^{-1}\}<q\leq \infty$ if $0<s\leq \frac{1}{2}$ and $\frac{2}{2s-1}<q\leq +\infty$ if $\frac{1}{2}<s<1$, let  $0\leq \zeta_0\in   L^q(D)$ be a function  with $0<|\mathrm{spt}(\zeta_0)|<\infty$ and $\mathcal W>0$ be a given constant. Assume that $$\emptyset\not= \tilde \Sigma_{\zeta_0,\mathcal W}\subset \mathcal{R}(\zeta_0).$$ Suppose that $\{\zeta_n\}_{n=1}^\infty\subset \mathcal{R}_+(\zeta_0)$ is a maximizing sequence in the sense that
	 	\begin{equation}\label{4-5}
	 		\tilde{E}_{\mathcal{W}} (\zeta_n)\rightarrow S_{\zeta_0,\mathcal{W}},\quad \text{as}\ \ n\rightarrow \infty.
	 	\end{equation}
	 	Then, there exist $\zeta^0\in \tilde\Sigma_{\zeta_0,\mathcal{W}}$, a subsequence $\{\zeta_{n_k}\}_{k=1}^\infty$ and a sequence of real numbers $\{c_k\}_{k=1}^\infty$ such that as $k\rightarrow \infty$,
	 	\begin{equation}\label{4-6}
	 		\|\zeta_{n_k}(\cdot+c_k\mathbf{e}_2)- \zeta^0\|_2\to 0.
	 	\end{equation}
	 \end{proposition}
	 \begin{proof}
	 	Note that since $0\in \overline{\mathcal{R}(\zeta_0)^w}\setminus  \mathcal{R}(\zeta_0)$, the condition $\emptyset\not= \tilde\Sigma_{\zeta_0,\mathcal W}\subset \mathcal{R}(\zeta_0) $ implies that $0$ is not a maximizer and hence we have $S_{\zeta_0,\mathcal{W}}>0$.
	 	
	 	Take $u_n= \zeta_n^2$.  Since  $0\leq \int_{D} u_n dx\leq \|\zeta_0\|_2^2<\infty$,  we may assume that, up to a subsequence (still denoted by $\{u_n\}_{n=1}^\infty$),  $$\int_{D} u_n dx\to \mu$$ for some $0\leq \mu\leq  \|\zeta_0\|_2^2$.  Applying Lemma \ref{lem4-6},  we find that for a certain subsequence, still denoted by $\{\zeta_n\}_{n=1}^\infty$, one of the three cases in Lemma \ref{lem4-6} should occur. In what follows,  we divide the proof into three steps.
	 	
	 	\emph{Step 1. Vanishing excluded:} 
	 	Suppose that for each fixed $R>0$,
	 	\begin{equation}\label{3-10}
	 		\lim_{n\rightarrow \infty}\sup_{y\in D}  \int_{B_R(y)\cap D}  \zeta^2_n dx =0.
	 	\end{equation}
	   By the property of rearrangement and H\"older's inequality, we have for any $R>0$ and $1\leq \tau \leq 2$
	 	$$\int_{B_R(y)\cap D}  \zeta^\tau_n dx \to 0$$ as $n\to+\infty$ uniformly over $y\in D$. 
	 	On the other hand, we have $\mathcal{G}_s^+(\zeta_n (1-1_{B_R(y)}))(y)\leq C\frac{\|\zeta_n\|_1 }{R^{2-2s}}$.
	 	Therefore, we get $$\int\zeta_n\mathcal{G}_s^+\zeta_n \leq \frac{C  }{R^{2-2s}}+ o_n(1),$$ for any $R>0$ and hence $\lim_{n\rightarrow \infty} \tilde E_{\mathcal W}(\zeta_n)\leq 0$. This is a contradiction to $S_{\zeta_0,\mathcal{W}}>0$. Thus, vanishing can not occur.

	 	\emph{Step 2. Dichotomy excluded:}
	 	We may assume that $\zeta_n$ is supported in $[0, M_0]\times \mathbb{R}$, where $M_0$ is the constant obtained in Lemma \ref{lem4-4}. Suppose that there is a constant $\alpha\in (0, \mu)$ such that for any $\varepsilon>0$, there exist $N(\varepsilon)\geq 1$ and $0\leq \zeta_{i,n}\leq \zeta_n, \,i=1,2,3$ satisfying
	 	\begin{equation*}
	 		\begin{cases}
	 			\zeta_n=\zeta_{1,n}+\zeta_{2,n}+\zeta_{3,n},\\
	 			\int_{D}  \zeta^2_{3,n}dx +|\alpha-\alpha_n|+|\mu-\alpha-\beta_n|<\varepsilon,\quad \text{for}\,\,n\geq N(\varepsilon),\\
	 			d_n:=\text{dist}(\text{spt}(\zeta_{1,n}), \text{spt}(\zeta_{2,n}))\rightarrow \infty, \quad \text{as}\,\,n\rightarrow \infty,
	 		\end{cases}	
	 	\end{equation*}
	 	where $\alpha_n=\int_{D}  \zeta^2_{1,n}dx$ and $ \beta_n=\int_{D}  \zeta^2_{2,n}dx$. Using a diagonal argument, we obtain that there exists a subsequence, still denoted by $\{\zeta_n\}_{n=1}^\infty$, such that
	 	\begin{equation*}
	 		\begin{cases}
	 			\zeta_n=\zeta_{1,n}+\zeta_{2,n}+\zeta_{3,n}, \quad 0\leq \zeta_{i,n}\leq \zeta_n, \,i=1,2,3\\
	 			\int_{D}  \zeta^2_{3,n}dx +|\alpha-\alpha_n|+|\mu-\alpha-\beta_n|\rightarrow 0,\quad \text{as}\,\,n\rightarrow \infty,\\
	 			d_n =\text{dist}(\text{spt}(\zeta_{1,n}), \text{spt}(\zeta_{2,n}))\rightarrow \infty, \quad \text{as}\,\,n\rightarrow \infty.
	 		\end{cases}	
	 	\end{equation*}
	 	
	 	By direct calculations, one has\begin{small}
	 		\begin{align*}
	 			&\int_{D}  \zeta_n\mathcal{G}_s^+ \zeta_n=\int_{D} (\zeta_{1,n}+\zeta_{2,n}+\zeta_{3,n})\mathcal{G}_s^+   (\zeta_{1,n}+\zeta_{2,n}+\zeta_{3,n})\\
	 			=&\int_{D}  \zeta_{1,n} \mathcal{G}_s^+\zeta_{1,n} +\int_{D}  \zeta_{2,n} \mathcal{G}_s^+\zeta_{2,n}  +2\int_{D}  \zeta_{1,n}\mathcal{G}_s^+\zeta_{2,n} +\int_{D}  \zeta_{3,n}\mathcal{G}_s^+(2\zeta_n-\zeta_{3,n} ).
	 	\end{align*}\end{small}
	 	By \eqref{2-6-1} and H\"older's inequality, we derive
	 	\begin{align*}
	 			\int_{D}  \zeta_{3,n}\mathcal{G}_s^+(2\zeta_n-\zeta_{3,n} )
	 			\leq  C  \|\zeta_{3,n}\|_{2-s} \|2\zeta_n-\zeta_{3,n}\|_{2-s}^{1-s}\|2\zeta_n-\zeta_{3,n}\|_1^s \\
	 			\leq C|\text{spt}(\zeta_0)|^{\frac{s}{2(2-s)}}\|\zeta_{3,n}\|_2 (\|\zeta_0\|_1+\|\zeta_0\|_2) =o_n(1).
	 	\end{align*}
	 	It is obvious that
	 	\begin{align*}
	 		\int_{D}\int_{D} \zeta_{1,n}(x)G_s^+(x,y)\zeta_{2,n}(y)dxdy\leq \frac{ C \|\zeta_0\|_1^2}{d_n^{2-2s}}=o_n(1).
	 	\end{align*}
	 	Hence, we arrive at
	 	$$\tilde{E}_{\mathcal W}(\zeta_n)=\frac{1}{2}\int_{D}  \zeta_n\mathcal{G}_s^+ \zeta_n-\mathcal W\int_{D} x_1\zeta_n dx \leq \tilde{E}_{\mathcal W}(\zeta_{1,n})+\tilde{E}_{\mathcal W}(\zeta_{2,n})+o_n(1).$$
	 	
	 	Taking Steiner symmetrization in the $x_2$-variable $\zeta^*_{i,n}$ of $\zeta_{i,n}$ for $i=1,2$, by the rearrangement inequality, we obtain
	 	\begin{equation*}
	 		\tilde{E}_{\mathcal W}(\zeta_n)\leq \tilde{E}_{\mathcal W}(\zeta^*_{1,n})+\tilde{E}_{\mathcal W}(\zeta^*_{2,n})+o_n(1).
	 	\end{equation*}
	 	By Lemma \ref{lem4-2}, there exists a constant $N_0>0$ depending on $\|\zeta_0\|_1+ \|\zeta_0\|_q$, $\mathcal W $ and $M_0$ such that for all $\zeta\in\mathcal{R}_+(\zeta_0)$ with $\text{spt}(\zeta)\subset ([0,M_0]\times \mathbb{R})\cap D$, $$\mathcal{G}_s^+\zeta(x)- \mathcal W   x_1\leq 0,\quad \forall x\ \text {with}\  |x_2|>N_0$$
	 	Let $$\zeta^{**}_{i,n}(x)=\zeta^*_{i,n}1_{[0,M_0]\times[-N_0, N_0]} (x+(-1)^i N_0 \mathbf{e}_2),\quad i=1,2.$$
	 	Then, we find
	 	\begin{equation*}
	 		\tilde{E}_{\mathcal W}(\zeta_n)\leq \tilde{E}_{\mathcal W}(\zeta^{**}_{1,n})+\tilde{E}_{\mathcal W}(\zeta^{**}_{2,n})+o_n(1).
	 	\end{equation*}
	 	and
	 	$$\text{supp}(\zeta^{**}_{1,n})\subset [0,M_0]\times[0, 2N_0],\  \  \ \text{supp}(\zeta^{**}_{2,n})\subset [0,M_0]\times[- 2N_0,0].$$
	 	We may assume that $\zeta^{**}_{i,n}\rightarrow \zeta^{**}_{i}$   weakly in $L^r(D)$ for some $s^{-1}<r<q$ and $i=1,2$. Then, $\zeta^{**}:=\zeta^{**}_{1}+\zeta^{**}_{2}\in \overline{\mathcal{R}(\zeta_0)^w}$. Moreover, by the weak convergence, we get
	 	$$\lim_{n\rightarrow \infty} \tilde E_{\mathcal W}(\zeta^{**}_{i,n})=\tilde E_{\mathcal W} (\zeta^{**}_{i}),\,\,\text{for}\, i=1,2,$$
	 	and therefore we arrive at
	 	$$\tilde E_{\mathcal W}(\zeta^{**}_{1})+\tilde E_{\mathcal W}(\zeta^{**}_{2})\geq \limsup_{n\to\infty}\tilde E_{\mathcal W}(\zeta_n)=S_{\zeta_0, \mathcal W}.$$
	 	It can be seen that
	 	\begin{equation*}
	 		\begin{split}
	 			&\quad S_{\zeta_0, \mathcal W}\geq \tilde E_{\mathcal W}(\zeta^{**})=\tilde E_{\mathcal W}(\zeta^{**}_{1}+\zeta^{**}_{2})\\
	 			&=\tilde E_{\mathcal W}(\zeta_1^{**})+\tilde E_{\mathcal W}(\zeta_2^{**})+\int_{D}\int_{D} \zeta_1^{**}(x)G_s^+(x,y)\zeta_2^{**}(y)dxdy\\
	 			&\geq S_{\zeta_0, \mathcal W}+\int_{D}\int_{D} \zeta_1^{**}(x)G_s^+(x,y)\zeta_2^{**}(y)dxdy,
	 		\end{split}
	 	\end{equation*}
	 	from which we must have
	 	\begin{equation}\label{4-8}
	 		\tilde{E}_{\mathcal W}(\zeta^{**})=S_{\zeta_0, \mathcal W}\ \ \  \text{and}\ \ \  \int_{D}\int_{D} \zeta_1^{**}(x)G_s^+(x,y)\zeta_2^{**}(y)dxdy=0.
	 	\end{equation}
	 	Since $\tilde \Sigma_{\zeta_0,\mathcal W}\subset \mathcal{R}(\zeta_0)$ by the assumption,  we deduce $\zeta^{**}\in  \mathcal{R}(\zeta_0)$ and $\mu=\|\zeta^{**}\|^2_2=\|\zeta_1^{**}\|^2_2+\|\zeta_2^{**}\|^2_2$.
	 	
	 	On the other hand, since  $\|\zeta_1^{**}\|^2_2\leq \alpha$ and $\|\zeta_2^{**}\|^2_2\leq \mu-\alpha$ by the weak convergence, we find that  $\|\zeta_1^{**}\|^2_2= \alpha>0$ and $\|\zeta_1^{**}\|^2_2=\mu-\alpha>0$, which implies that both $\zeta_1^{**}$ and $\zeta_2^{**}$ are non-zero and hence $\int_{D}\int_{D} \zeta_1^{**}(x)G_s^+(x,y)\zeta_2^{**}(y)dxdy>0$, which is a contradiction to \eqref{4-8}.

	 	\emph{Step 3. Compactness:} Assume that there is a sequence $\{y_n\}_{n=1}^\infty$ in $\overline{D}$ such that for arbitrary $\varepsilon>0$, there exists $R>0$ satisfying
	 	\begin{equation}\label{4-9}
	 		\int_{D\cap B_R(y_n)}  \zeta_n^2 dx\geq \mu-\varepsilon, \quad \forall\, n\geq 1.
	 	\end{equation}
	 	We may assume that $y_n=(  y_{n,1}, 0)$ after a suitable translation in $x_2$-variable.
	 	Define $\zeta_n^0:=\zeta_n 1_{(0, M_0)\times \mathbb R}$ and $\zeta_n^R:=\zeta_n 1_{(0, M)\times (-R, R)}$. Then $\{\zeta_n^0\}_{n=1}^\infty$ is also a maximizing sequence in $\mathcal{R}_+(\zeta_0)$ by Lemma \ref{lem4-3}. Moreover, we infer from \eqref{4-9} that for arbitrary $\ep>0$, there exists $R>0$ such that $$\|\zeta_n^0-\zeta_n^R\|_2^2\leq \ep,\quad \forall\  n\geq 1.$$
	 	That is,
	 	\begin{equation}\label{4-10}
	 		\|\zeta_n^0-\zeta_n^R\|_2\to 0,\quad \text{as}\ R\to \infty, \quad \text{uniformly over} \  n.
	 	\end{equation}
	 	
	 	We may assume that $\zeta_n^0\to \zeta^0$ weakly in $L^2(D)\cap L^r(D)$ for some $s^{-1}<r<q$ and hence $\zeta_n^R\to \zeta^01_{(0, M_0)\times (-R, R)}$ weakly in $L^2(D)\cap L^r(D)$.
	 	By the weak convergence and $\zeta_n\in \mathcal{R}_+(\zeta_0)$, we find
	 	\begin{equation}\label{4-11}
	 		\|\zeta^0\|_2\leq \liminf_{n\to\infty}\|\zeta_n^0\|_2\leq \|\zeta_0\|_2.
	 	\end{equation}
	 	
	 	Using Lemma \ref{lem2-2} and H\"older's inequality, we conclude
	 	$|E(\zeta_n^0)-E(\zeta_n^R)|=o_R(1)$. That is, $E(\zeta_n^R)\to E(\zeta_n^0)$ as $R\to \infty$ uniformly over $n$ by \eqref{4-10}. On the other hand $E(\zeta_n^R)\to E(\zeta^R)$ as $n\to \infty$ for fixed $R$ by weak continuity of $E$ in functions supported on  bounded domains and $E(\zeta^R)\to E(\zeta^0)$ by the monotone convergence theorem. Therefore, we obtain $$E(\zeta_n^0)\to E(\zeta^0).$$
	 	
	 	As for the impulse, we split
	 	\begin{equation*}
	 		|I(\zeta_n^0)-I(\zeta^0)|\leq |I(\zeta_n^0)-I(\zeta_n^R)|+|I(\zeta_n^R)-I (\zeta^R)|+|I(\zeta^R)-I(\zeta^0)|.
	 	\end{equation*}
	 	For the first term, by H\"older's inequality, we deduce
	 	$$|I(\zeta_n^0)-I(\zeta_n^R)|\leq M_0 |\text{spt}(\zeta_n^0)|^{\frac{1}{2}}\|\zeta_n^0-\zeta_n^R\|_2\to 0, $$
	 	as $R\to \infty$ uniformly over $n$.
	 	For fixed $R$, we have the second term $|I(\zeta_n^R)-I(\zeta^R)|\to 0$ as $n\to \infty$ by the weak convergence. Since the third term $|I(\zeta^R)-I(\zeta^0)|\to0$ as $R\to \infty$ by  the monotone convergence theorem, we have $|I(\zeta_n^0)-I(\zeta^0)|\to 0$ by first letting $R\to \infty$ and then $n\to\infty$.
	 	
	 	Therefore, we have proved $\tilde{E}_{\mathcal W}(\zeta_n^0)\to \tilde{E}_{\mathcal W}(\zeta^0)$ and hence $\tilde{E}_{\mathcal W}(\zeta^0)=S_{\zeta_0,\mathcal W}$ and $\zeta^0\in \mathcal{R}(\zeta_0)$ by our assumption $\Sigma_{\zeta_0,\mathcal{W}}\subset \mathcal{R}(\zeta_0)$. Then we deduce that  $	\|\zeta^0\|_2=\|\zeta_0\|_2$ by the property of rearrangement, which implies $\lim_{n\to\infty}\|\zeta_n^0\|_2=\|\zeta^0\|_2$. So, we obtain the strong convergence $\|\zeta_n^0-\zeta^0\|_2\to 0$ by the uniform convexity of $L^2(D)$.
	 	
	 	Now we want to show that $\zeta_n\to \zeta^0$ strongly. Indeed, since the supports of $\zeta_n^0$ and $\zeta_n- \zeta^0_n$ are disjoint and $\zeta_n\in \mathcal{R}(\zeta_0)$, we conclude
	 	$$\|\zeta_n-\zeta_n^0\|_2^2=\|\zeta_n\|_2^2-\| \zeta_n^0\|_2^2\leq \|\zeta_0\|_2^2-\| \zeta_n^0\|_2^2\to \|\zeta_0\|_2^2-\| \zeta^0\|_2^2=0.$$
	 	Therefore, we obtain $\|\zeta_n-\zeta^0\|_2\to 0$ and finish the proof.	 	
	 \end{proof}	
	
	 To state our stability result, following \cite{BNL13}, we need to introduce some definitions first.
	 \begin{definition}\label{def2}
	 	Let $\xi_0\in L^1\cap L^p(\mathbb R_+^2)$ be a given function.   $\xi\in L^\infty_{loc}(\left[0,T\right), L^1(\mathbb{R}^2_+))\cap  L^\infty_{loc}(\left[0,T\right), L^p(\mathbb{R}^2_+))$ is called a \emph{ $L^p$-regular solution} in $(0, T)$ of \eqref{1-1} with initial data $\xi_0$ if
	 	\begin{itemize}
	 		\item [(i)] $\bar\xi(x,t):=\xi(x,t)-\xi(\bar x,t )$ satisfies \eqref{1-1}
	 		in the sense of distributions with initial data $\bar\xi_0(x):=\xi_0(x)-\xi_0(\bar x)$;
	 		\item [(ii)] $E(\xi(t,\cdot))$ and $I(\xi(t,\cdot))$  are constant and $\xi(t)\in \mathcal{R}(\xi_0)$ for $t\in[0, T)$;
	 		\item [(iii)] $\xi $ is   non-negative for $t\in(0, T)$ provided that $\xi_0$ is non-negative;
	 		\item [(iv)] For $0<s\leq \frac{1}{2}$, we require that $\xi $ is  supported in $D$ for $t\in(0, T)$ provided that $\xi_0$ is  supported in $D$.
	 	\end{itemize}
	 \end{definition}

 Generally speaking, an $L^p$-regular solution is a weak solution to \eqref{1-1} such that its kinetic energy, impulse and distribution  are conserved, which is true for sufficiently regular solutions; see \cite{BSV19} for some discussion about the conservation laws.  The existence of $L^p$-regular solutions for the Euler equation was obtained in \cite{BNL13} using the transport nature of the Euler equation, see also \cite{Abe}. Note that the gSQG equations are also transport equations. So it is possible to modify the method in \cite{Abe,BNL13} to show the existence of $L^p$-regular solutions defined above for the gSQG equation  provided that the existence of sufficiently smooth solutions for the Cauchy problem of the gSQG equation is a priori known. Thus, the $L^p$-regular solutions for the gSQG equation may be proved to exist for small $T$ due to the local well-posedness theory. For large $T$, though a general theory for the global well-posedness for the Cauchy problem of the gSQG equation remains a challenging open problem now, various global solutions known as relative equilibria are constructed as mentioned in the introduction, which are trivial examples of the $L^p$-regular solutions. Invariant measures and a large class of global solutions for the SQG equation are also obtained in \cite{Fol}. Besides, certain blow-up scenarios have been ruled out analytically in \cite{Cor1, CorF} and numerically in \cite{Con2}. Therefore, we believe that the $L^p$-regular solutions are reasonable to exist for large $T$ and for a large class of initial values.

 We say that an initial data $\xi_0$ is \emph{admissible} if $\xi_0$ is nonnegative, supported in $D$ if $0<s\leq \frac{1}{2}$ and there exists a $L^\infty$-regular solution with initial data $\xi_0$ for $T=+\infty$.
	
	 Now, we are ready to establish the following stability theorem on the set of maximizers.
	 \begin{theorem}\label{Sset}
	 	For $q$ with $\max\{2, s^{-1}\}<q\leq \infty$ if $0<s\leq \frac{1}{2}$ and $\frac{2}{2s-1}<q\leq +\infty$ if $\frac{1}{2}<s<1$, let  $0\leq \zeta_0\in L^q(D)$ be a function  with $0<|\mathrm{spt}(\zeta_0)|<\infty$ and $\mathcal W>0$ is a given constant.
	 	Suppose that $\tilde\Sigma_{\zeta_0,\mathcal W}$,  the set of maximizers of $\tilde{E}_{\mathcal W}:=E- \mathcal W  I$ over $\overline{\mathcal{R}(\zeta_0)^w}$, satisfies $$\emptyset\not= \tilde \Sigma_{\zeta_0,\mathcal W}\subset \mathcal{R}(\zeta_0), $$ and all elements of $\tilde \Sigma_{\zeta_0,\mathcal W}$ are admissible.  Then $\tilde \Sigma_{\zeta_0,\mathcal W}$ is orbitally stable in the following sense:
	 	
	 	For arbitrary $\eta>0$ and $M>0$, there exists $\delta>0$ such that   if there exist a $L^\infty$-regular solution $\xi(t)$ in $(0,T)$ with initial data $\xi_0$,  $\|\xi_0\|_\infty\leq M$  and
	 	\begin{equation*}
	 		\inf_{\zeta\in\tilde \Sigma_{\zeta_0,\mathcal W}}\{ \|\xi_0-\zeta\|_1+ \|\xi_0-\zeta\|_2+|I(\xi_0-\zeta)| \} \leq \delta,
	 	\end{equation*}
	 	then for all $t\in (0, T)$, we have
	 	\begin{equation}\label{s-1}
	 		\inf_{\zeta\in \tilde\Sigma_{\zeta_0,\mathcal W}} \left\{  \|\xi(t)-\zeta\|_1+\|\xi(t)-\zeta\|_2\right\}\leq \eta.
	 	\end{equation}
	
	 	If in addition $$I(\zeta_1)=I(\zeta_2),\quad \forall\  \zeta_1, \zeta_2\in \tilde \Sigma_{\zeta_0,\mathcal W}, $$
	 	then for all $t\in (0,T)$, we have
	 	\begin{equation}\label{s-2}
	 		\inf_{\zeta\in \tilde \Sigma_{\zeta_0,\mathcal W}} \left\{\|\xi(t)-\zeta\|_1+\|\xi(t)-\zeta\|_2+I(|\xi(t)-\zeta|)\right\}\leq \eta.
	 	\end{equation}
	 	
	 \end{theorem}	
	 \begin{proof}
	 	By Lemma \ref{lem4-4}, any maximizer (if exists) is supported in $[0, M_0]\times \mathbb R$.  So, we deduce that $\sup_{\zeta_1\in \tilde \Sigma_{\zeta_0,\mathcal W}} I(\zeta_1)\leq M_0  \|\zeta_0\|_1<\infty.$
	 	By   an  argument similar to the proof of Lemma \ref{lem4-4}, for $\zeta$ with the following properties:
	 	\begin{equation*}
	 		\zeta\geq 0,\ \ \ \|\zeta\|_1 \leq  \|\zeta_0\|_1+1, \ \ \  \|\zeta\|_\infty\leq  M,
	 	\end{equation*}
	 	there exists a constant $M_1\geq M_0$ independent of $\zeta$ such that
	 	$$\mathcal{G}_s^+\zeta(x)- \mathcal W  x_1 <0,\quad \forall x\ \text {with}\  x_1>M_1.$$
	 	
	 	Now we finish our proof  by  contradiction. Suppose on the contrary that there exist a sequence of non-negative functions $\{\omega_0^n\}_{n=1}^\infty$, each of which is admissible,  $\|\omega_0^n\|_\infty\leq M$ and	as $n\to \infty$,
	 	\begin{equation*}
	 		\inf_{\zeta\in \tilde\Sigma_{\zeta_0,\mathcal W}} \{\|\omega_0^n-\zeta\|_1+\|\omega_0^n-\zeta\|_2 +|I(\omega_0^n-\zeta)|\}\to 0,
	 	\end{equation*}
	 	while
	 	\begin{equation*}
	 		\sup_{t\in (0,T_n)}\inf_{\zeta\in \tilde \Sigma_{\zeta_0,\mathcal W}} \left\{  \|\omega^n(t)-\zeta\|_1+\|\omega^n(t)-\zeta\|_2\right\}\geq c_0>0,
	 	\end{equation*}
	 	for some positive constant $c_0$, where $\omega^n(t)$ is the $L^\infty$-regular  solution of \eqref{1-1} with initial data $\omega_0^n$  in $(0,T_n)$. By the choice of $\omega_0^n$, we infer from Lemma \ref{lem2-2} and the H\"older inequality that	as $n\to \infty$,
	 	\begin{equation}\label{4-14}
	 		\tilde{E}_{\mathcal W}(\omega_0^n)\to S_{\zeta_0,\mathcal W}.
	 	\end{equation}
	 	We  choose $t_n\in(0,T_n)$ such that for each $n$,
	 	\begin{equation}\label{4-15}
	 		\inf_{\zeta\in\tilde \Sigma_{\zeta_0,\mathcal W}} \left\{ \|\omega^n(t_n)-\zeta\|_1+\|\omega^n(t_n)-\zeta\|_2 \right\}\geq \frac{c_0}{2}>0.
	 	\end{equation}

	 	Take a sequence of functions $\zeta_0^n \in \tilde\Sigma_{\zeta_0,\mathcal W} \subset \mathcal{R}(\zeta_0)$ such that 	as $n\to \infty$,
	 	\begin{equation}\label{4-16}
	 		\|\omega_0^n-\zeta_0^n\|_1+\|\omega_0^n-\zeta_0^n\|_2+|I(\omega_0^n-\zeta_0^n)|\to 0.
	 	\end{equation}
	 	Then, $\{\zeta_0^n\}_{n=1}^\infty$ is a maximizing sequence due to Lemma \ref{lem2-2}.
	 	
	 	For a function $\omega$, we use $\bar{\omega}:=\omega 1_{(0,M_1)\times \mathbb{R}}$ to denote the restriction of $\omega$. By Lemma \ref{lem4-3}, the conservation of energy and impulse, we conclude
	 	\begin{equation}\label{4-17}
	 		\tilde{E}_{\mathcal W}(\overline{\omega}^n(t_n))\geq 	\tilde{E}_{\mathcal W}( \omega^n(t_n))=	\tilde{E}_{\mathcal W}( \omega^n_0)\to S_{\zeta_0,\mathcal W}.
	 	\end{equation}
	 	Since $\omega^n(t_n)$ is a rearrangement of $\omega^n_0$ by Definition \ref{def2}, one can find a rearrangement $\zeta_1^n\in \mathcal R(\zeta_0^n)=\mathcal R(\zeta_0)$ such that as $n\to+\infty$
	 	\begin{equation}\label{4-16-2}
	 	\|\omega^n(t_n)-\zeta_1^n\|_1+\|\omega^n(t_n)-\zeta_1^n\|_2 =	\|\omega_0^n-\zeta_0^n\|_1+\|\omega_0^n-\zeta_0^n\|_2 \to 0.
	 	\end{equation}
	  Then, we infer from  \eqref{2-6-1}, \eqref{4-16-2} and H\"older's inequality that
	 	\begin{equation}\label{4-18}
	 		\begin{split}
	 			|I(\bar\zeta_1^n )-I(\bar\omega^n(t_n))|&\leq M_1  \|\bar\zeta_1^n-\bar\omega^n(t_n)\|_1\\
	 			&\leq M_1  \| \zeta_1^n- \omega^n(t_n)\|_1\\
	 			&=  M_1  \| \zeta^n_0- \omega^n_0\|_1,
	 		\end{split}
	 	\end{equation}
	 	and
	 	\begin{equation}\label{4-19}
	 		\begin{split}
	 			|E(\bar\zeta_1^n )-E(\bar\omega^n(t_n))|&\leq C   \|\bar\zeta_1^n-\bar\omega^n(t_n)\|_{2-s}^{1-s}\leq C \|\zeta^n_0- \omega^n_0\|_{2-s}^{1-s}.
	 		\end{split}
	 	\end{equation}
	 	So, we deduce from \eqref{4-16}--\eqref{4-19} that $\{ \bar\zeta_1^n\}\subset \mathcal{R}_+(\zeta_0)$ is a maximizing sequence and hence by Proposition \ref{prop4-7} after some translations, we have
	 	\begin{equation}\label{4-20}
	 		\bar\zeta_1^n\to \zeta_{**} \ \  \text{strongly in}\ \  L^2(D),
	 	\end{equation}
	 	as $n\to \infty$ for some function $ \zeta_{**} \in \tilde\Sigma_{\zeta_0,\mathcal W}$, which implies $\zeta_1^n\to  \zeta_{**}$ strongly. Indeed, since the supports of $\bar\zeta_1^n$ and $\zeta_1^n- \bar\zeta_1^n$ are disjoint and $\zeta_1^n\in \mathcal{R}(\zeta_0)$, we conclude
	 	$$\|\zeta_1^n-\bar\zeta_1^n\|_2^2=\|\zeta_1^n\|_2^2-\|\bar\zeta_1^n\|_2^2\leq \|\zeta_0\|_2^2-\| \bar\zeta_1^n\|_2^2\to \|\zeta_0\|_2^2-\|\zeta_{**}\|_2^2=0,$$
	 	from which we get
	 	$$\|\zeta_1^n-\zeta_{**}\|_2\leq \|\zeta_1^n-\bar\zeta_1^n\|_2+\|\zeta_{**}-\bar\zeta_1^n\|_2\to 0.$$
	 	Hence, by \eqref{4-16-2} and  \eqref{4-20}, we deduce
	 	\begin{equation*}
	 		\begin{split}
	 			\quad\|\omega^n(t_n)-\zeta_{**}\|_2 &\leq  \|\zeta_1^n-\omega^n(t_n)\|_2+\|\zeta_1^n-\zeta_{**}\|_2\\
	 			&=\|\zeta^n_0-\omega^n_0\|_2+\|\zeta_1^n-\zeta_{**}\|_2=o_n(1).\end{split}
	 	\end{equation*}
	 	
	 	Next, we will estimate $\|\omega^n(t_n)-\zeta_{**}\|_1$.
	 	By the conservation of the $L^1$-norm and \eqref{4-20}, we have
	 	\begin{equation*}
	 		\begin{split}
	 			&\ \ \  \|\omega^n(t_n)-\zeta_{**}\|_1 \leq  \int_{\text{spt}(\zeta_{**})}  |\omega^n(t_n)-\zeta_{**}|dx+\int_{D\setminus\text{spt}(\zeta_{**})}   \omega^n(t_n) dx\\
	 			&\leq \int_{\text{spt}(\zeta_{**})}  |\omega^n(t_n)-\zeta_{**}|dx+\int_{D }   \omega^n(t_n) dx-\int_{D}  \zeta_0^n dx+\int_{D}   \zeta_{**} dx-\int_{\text{spt}(\zeta_{**})  }  \omega^n(t_n) dx\\
	 			&\leq 2\int_{\text{spt}(\zeta_{**})}  |\omega^n(t_n)-\zeta_{**}|dx+\int_{D }|\omega_0^n-\zeta_0^n| dx \\
	 			&\leq 2  |\text{supp}(\zeta_0)|^{\frac{1}{2}}  \| \omega^n(t_n)-\zeta_{**}\|_2+ \int_{D}|\omega_0^n-\zeta_0^n| dx=o_n(1),
	 		\end{split}
	 	\end{equation*}
	 	which contradicts the choice of $t_n$ in \eqref{4-15} and completes the proof of \eqref{s-1}. Here we have used $\int_{D}  \zeta_0^n dx=\int_{D}   \zeta_{**} dx= \int_{D}  \zeta_0  dx$ and $ |\text{spt}(\zeta_0^n)|= |\text{spt}(\zeta_{**}) |=|\text{spt}(\zeta_0)|$ due to the properties of rearrangement.
	 	
	 	If in addition $$I(\zeta_1)=I(\zeta_2),\quad \forall\  \zeta_1, \zeta_2\in \tilde\Sigma_{\zeta_0,\mathcal W}, $$ then by the conservation of the impulse, we find
	 	\begin{footnotesize}
	 		\begin{equation*}
	 			\begin{split}
	 				&I( |\omega(t)-\zeta_1|) \leq  \int_{\text{spt}(\zeta_1)} x_1 |\omega(t)-\zeta_1|dx+\int_{D\setminus\text{spt}(\zeta_1)}  x_1 \omega(t) dx\\
	 				&\leq \int_{\text{spt}(\zeta_1)} x_1 |\zeta_1  -  \omega(t)| dx+\int_{D}x_1   \omega(t) dx-\int_{D} x_1  \zeta_2 dx+\int_{D} x_1  \zeta_1 dx-\int_{ \text{spt}(\zeta_1)}x_1   \omega(t) dx\\
	 				&\leq 2\int_{\text{spt}(\zeta_1)} x_1  |\zeta_1  -  \omega(t)| dx+\int_{D}x_1   \omega_0 dx-\int_{D} x_1  \zeta_2 dx\\
	 				&\leq 2 M_0 |\text{spt}(\zeta_0)|^{\frac{1}{2}} 	\| \omega(t)-\zeta_1\|_2+ |I( \omega_0-\zeta_2)|, \quad \forall \ \zeta_1, \zeta_2 \in \tilde \Sigma_{\zeta_0,\mathcal W},
	 			\end{split}
	 		\end{equation*}
	 	\end{footnotesize}
	 	which implies \eqref{s-2} by \eqref{s-1} and taking $\delta$ smaller if necessary. 	
	 \end{proof}
	
	\begin{remark}
		Our proof also works well for the case $s=1$.  Compared with Theorem 1 in \cite{BNL13},  we admit perturbations with non-compact supports. This is achieved by   bringing in   the $L^1$-norm in our theorem. Note that, if the perturbation $\omega_0$ has a compact support with measure less than $A$ for some constant $A$ as in Theorem 1 in \cite{BNL13} and is closed to $\tilde \Sigma_{\zeta_0, \mathcal W}$ in $L^2$, then the H\"older inequality will imply that $\omega_0$ is closed to $\tilde \Sigma_{\zeta_0, \mathcal W}$ in $L^1$. Thus, one can obtain  a generalization of Theorem 1 in \cite{BNL13} by the argument in this paper.
	\end{remark}

	 \subsection{Maximizers in rearrangement class and the stability of traveling-wave solutions}\label{sec4-2}
	
	 In Sections \ref{sec2} and \ref{sec3}, for $J(t)=L t^{1+\frac{1}{p}}$  for some $L>0$ and  $p\in(1, \frac{1}{1-s})$, we have obtained existence and uniqueness of a traveling solution $\omega_\ep$.
	 To apply our stability result   Theorem \ref{Sset} to obtain the stability of $\omega_\ep$,  we consider the following maximizing problem
	 \begin{equation}\label{4-21}
	 	\begin{split}
	 		\sup_{\zeta\in\overline{\mathcal{R}(\omega_\ep)^w}}\tilde E_\ep(\zeta),
	 	\end{split}
	 \end{equation}
	 where (with abuse of notations for simplicity) $$\tilde E_{\ep}(\zeta)	:=\tilde E_{W \ep^{3-2s}}(\zeta)=\frac{1}{2}\int_{D}{\zeta \mathcal{G}_s^+\zeta}dx-W \ep^{3-2s}\int_{D}x_1 \zeta dx.$$
	 We    denote   $  \tilde\Sigma_\ep$ as the set of maximizers of $\tilde E_\ep$ over the rearrangement class $\overline{\mathcal{R}(\omega_\ep)^w}$. Our main result in this subsection is the following theorem.
	 \begin{theorem}\label{thm4-10}
	 	For $\ep>0$ small, the set of maximizers $\tilde\Sigma_\ep$ satisfies $$\emptyset\neq  \tilde\Sigma_\ep\subset \mathcal{R}(\omega_\ep).$$ Moreover, for each  $\zeta^\ep \in   \tilde\Sigma_\ep$,  the following claims hold.
	 	\begin{itemize}
	 		\item [(i).]$ {\zeta}^\varepsilon= g_\ep(\mathcal{G}_s^+\zeta^\ep- W\ep^{3-2s}x_1 -\tilde\mu_\ep)$ for some non-negative and non-decreasing function $g_\varepsilon:\mathbb{R}\to \mathbb{R}$ satisfying $g_\varepsilon(\tau)>0$ if $\tau>0$ and $g_\varepsilon(\tau)=0$ if $\tau\le0$ and some constant $\tilde \mu_\ep$.
	 		\item [(ii).]  $   \text{diam}\left(\mathrm{spt}({\zeta}^\varepsilon)\right)<\tilde R\ep$ for some constant  $0<\tilde R<\infty$ and up to a suitable translation in the $x_2$ direction $$\sup_{x\in \mathrm{spt}({\zeta}^\varepsilon)} |\ep x-(d_0,0)|=o(1).$$\end{itemize}
	 \end{theorem}
	
	 In order to prove Theorem \ref{thm4-10}, a series of lemmas is needed.
	 \begin{lemma}\label{lem4-11}
	 	The energy satisfies 
	 	\begin{equation*}
	 		\max_{\zeta\in\overline{\mathcal{R}(\omega_\ep)^w}}\tilde E_\varepsilon(\zeta)\geq I_0+O(\ep^{2-2s}).
	 	\end{equation*}
	 \end{lemma}
	 \begin{proof}
	 	This is a simple consequence of the fact $\omega_\ep\in\overline{\mathcal{R}(\omega_\ep)^w}$ and Lemma \ref{lem2-20}.
	 \end{proof}
	
	 \begin{lemma}\label{lem4-12}
	 	$\tilde{E}_\ep$ attains its maximum value over $\overline{\mathcal{R}(\omega_\ep)^w}$ at some $\zeta^\varepsilon$, which is Steiner symmetric in the $x_2$-variable.
	 \end{lemma}
	 \begin{proof}
	 	By Lemma \ref{lem4-1}, there exists a constant $1<M_\ep=O(\ep^{-1})$ depending  on $\ep$, $\omega_\ep$ and $W$ such that
	 	$$\mathcal{G}_s^+\zeta(x)- W\ep^{3-2s}x_1 < 0,\quad \forall x\ \text {with}\  x_1>M_\ep, \quad \forall \zeta\in \overline{\mathcal{R}(\omega_\ep)^w}.$$
	 	It is easy to check that $\tilde {E}_\ep$ is bounded from above over $\overline{\mathcal{R}(\omega_\ep)^w}$ by Lemma \ref{lem2-2}. Take a sequence $\{\zeta_j\}\subset \overline{\mathcal{R}(\omega_\ep)^w}$   such that as $j\to+\infty$
	 	\begin{equation*}
	 		\tilde{E}_\ep(\zeta_j)\to \sup_{\overline{\mathcal{R}(\omega_\ep)^w}}\tilde{E}_\ep.
	 	\end{equation*}
	 	We may assume that $\zeta_j$ is supported in $(0, M_\ep)\times \mathbb R$ by Lemma \ref{lem4-3}. We can also assume $\zeta_j$ is Steiner symmetric about the $x_1$-axis by replacing $\zeta_j$ with its own  Steiner symmetrization. Since $\int_{D} x_1  \zeta_j dx\leq R_\ep \kappa$,  there exists a constant $N_\ep>0$ such that $\mathcal{G}_s^+\zeta_j(x)-  W\ep^{3-2s} x_1   <0,\quad \forall x\in D\ \text {with}\  |x_2|>N_\ep,\quad\forall \  j$ due to Lemma \ref{lem4-2}. That is, we assume that  $\zeta_j$ is supported in $[0, M_\ep]\times [-N_\ep,N_\ep]$. There is a subsequence (still denoted by $\{\zeta_j\}$) such that as $j\to +\infty$, $\zeta_j\to \zeta^\varepsilon \in \overline{\mathcal{R}(\omega_\ep)^w}$ weakly star in $L^{\infty}(D)$. Since $G_s^+(\cdot,\cdot)\in L^{r}_{loc}(\mathbb{R}_+^2\times \mathbb{R}_+^2) $ for any $1\leq r<\frac{1}{1-s}$, we deduce that
	 	\begin{equation*}
	 		\lim_{j\to +\infty}\tilde{E}_\ep(\zeta_j)=\tilde{E}_\ep(\zeta^\varepsilon).
	 	\end{equation*}
	 	This means that $\zeta^\varepsilon$ is a maximizer and thus  the proof is finished.
	 \end{proof}
	 To study the properties of maximizers. We need the following lemma from \cite{Bu1}.
	 \begin{lemma}[Lemmas 2.4 and 2.9 in \cite{Bu1}]\label{lem4-13}
	 	Let $(\Omega,\nu)$ be a finite positive measure space. Let $\xi_0: \Omega\to \mathbb{R}$ and $\zeta_0: \Omega\to \mathbb{R}$ be $\nu$-measurable functions, and suppose that every level set of $\zeta_0$ has zero measure. Then there is a non-decreasing function $f$ such that $f\circ \zeta_0$ is a rearrangement of $\xi_0$. Moreover, if $\xi_0 \in L^q(\Omega,\nu)$ for some $1\le q<+\infty$ and $\zeta_0\in L^{q'}(\Omega,\nu)$, then $f\circ \zeta_0$ is the unique maximizer of linear functional
	 	\begin{equation*}
	 		M(\xi):=\int_\Omega \xi(x)\zeta_0(x)d\nu(x)
	 	\end{equation*}
	 	relative to $\overline{\mathcal{R}(\xi_0)^w}$.
	 \end{lemma}
	
	 \begin{lemma}\label{lem4-14}
	 	Let  $\zeta^\varepsilon \in \tilde\Sigma_\ep$ be a maximizer. Then, up to a translation in the $x_2$ direction, $\zeta^\ep$ must be Steiner symmetric in the $x_2$-variable and supported in $[0, M_\ep]\times [-N_\ep,N_\ep]$.   Moreover,   for $\ep>0$ small,  it holds    $$0\not\equiv \zeta^\varepsilon\in \mathcal{R}(\zeta_\ep)$$ and there exists a non-negative and non-decreasing function $g_\varepsilon: \mathbb{R}\to \mathbb{R}$ satisfying $g_\varepsilon(\tau)>0$ if $\tau>0$  and $g_\varepsilon(\tau)=0$ if $\tau\le 0$, such that for some constant $\tilde\mu_\varepsilon$,
	 	\begin{equation}\label{4-22}
	 		\zeta^\varepsilon(x)=g_\varepsilon(\mathcal{G}_1\zeta^\varepsilon(x)-W \ep^{3-2s}x_1-\tilde\mu_\varepsilon),\ \ \forall\,x\in \mathbb{R}_+^2.
	 	\end{equation}
	 
	 \end{lemma}
	 \begin{proof}
	 	For $\ep>0$ small, we see from Lemma \ref{lem4-11} that $\tilde {E}_\ep(\zeta^\varepsilon)>0$ and hence $\zeta^\ep\not\equiv 0$.

	 	Since $\overline{\mathcal{R}(\omega_\ep)^w}$ is a convex set (see e.g. \cite{Bu0, BNL13} and references therein). Thus for each $\zeta\in \overline{\mathcal{R}(\omega_\ep)^w}$, it holds $\zeta_\tau:=\zeta^\varepsilon+\tau(\zeta-\zeta^\varepsilon)\in \overline{\mathcal{R}(\omega_\ep)^w}$ for any $\tau\in [0,1]$. Noting that  $\zeta^\varepsilon$ is a maximizer of $\tilde E_\ep$, we have
	 	\begin{equation*}
	 		\frac{d}{d\tau}\bigg|_{\tau=0^+}\tilde{E}_\ep(\zeta_\tau)=\int_{D}(\zeta-\zeta^\varepsilon)\left(\mathcal{G}_s^+\zeta^\varepsilon - W\ep^{3-2s}x_1 \right) dx\le 0,
	 	\end{equation*}
	 	which yields
	 	\begin{equation*}
	 		\int_{D}\zeta\left(\mathcal{G}_s^+\zeta^\varepsilon - W\ep^{3-2s}x_1\right) dx\leq \int_{D}\zeta^\varepsilon \left(\mathcal{G}_s^+\zeta^\varepsilon -W \ep^{3-2s}x_1\right) dx, \quad \forall \  \zeta\in \overline{\mathcal{R}(\omega_\ep)^w}.
	 	\end{equation*}
	 	
	 	By the strict rearrangement inequality, we know that after a translation in the $x_2$ direction, $\zeta^\ep$ must be Steiner symmetric in the $x_2$-variable. Then 	by Lemmas \ref{lem4-1}-- \ref{lem4-3}, we find that $\zeta^\ep$ is supported in $[0, M_\ep]\times [-N_\ep,N_\ep]$.  Using the fact that $\zeta^\varepsilon$ is Steiner symmetric in the $x_2$-variable, one can verified   that $\mathcal{G}_s^+\zeta^\varepsilon$ is even and strictly decreasing with respect to $x_2$. It follows  that every level set of $\mathcal{G}_s^+\zeta^\varepsilon- W\ep^{3-2s}x_1$  has measure zero. By  Lemma \ref{lem4-13}, there exists a non-decreasing function $\tilde{g}_\varepsilon: \mathbb{R}\to \mathbb{R}$, such that,
	 	\begin{equation*}
	 		\tilde\zeta^\varepsilon(x)=\tilde{g}_\varepsilon\left(\mathcal{G}_s^+\zeta^\varepsilon(x)- W\ep^{3-2s}x_1\right).
	 	\end{equation*}
	 	for some $\tilde\zeta^\varepsilon\in \mathcal{R}(\omega_\ep)$. From the conclusion of Lemma \ref{lem4-13}, we also know $\tilde\zeta^\varepsilon(x)$ is the unique maximizer of the linear functional $\zeta \mapsto \int_{[0, M_\ep]\times [-N_\ep,N_\ep]} \left(\mathcal{G}_s^+\zeta^\varepsilon - W\ep^{3-2s}x_1\right) \zeta dx$ relative to $\overline{\mathcal{R}(\omega_\ep)^w}$. Hence we must have $\zeta^\varepsilon=\tilde\zeta^\varepsilon\in \mathcal{R}(\omega_\ep)$.
	 	
	 	Now, let
	 	\begin{equation*}
	 		\tilde\mu_\varepsilon:=\sup\left\{\mathcal{G}_s^+\zeta^\varepsilon(x)-W \ep^{3-2s}x_1\bigg|\ x\in \mathbb{R}_+^2\ \ \text{s.t.}\  \zeta^\varepsilon(x)=0 \right\}\in \mathbb{R},
	 	\end{equation*}
	 	and $g_\varepsilon(\cdot)=\max\{\tilde{g}_\varepsilon(\cdot+\tilde\mu_\varepsilon),0\}$. We have $$\zeta^\varepsilon(x)=g_\varepsilon\left(\mathcal{G}_s^+\zeta^\varepsilon(x)- W\ep^{3-2s}x_1-\tilde\mu_\varepsilon\right) \quad \text{for any} \ x\in D.$$
	 	
	 	The proof is thus complete.
	 \end{proof}
	
	 \begin{lemma}\label{lem4-15}
	 	Let  $\zeta^\varepsilon \in \tilde\Sigma_\ep\subset \mathcal{R}(\omega_\ep)$ be a maximizer. Then there exists correspondingly $z_\ep\subset [-R, M_\ep+R]\times \{0\}$ such that for $\ep>0$ small $$\|\zeta^\ep-\tilde\omega_\ep(\cdot-z_\ep)\|_1+\|\zeta^\ep-\tilde\omega_\ep(\cdot-z_\ep)\|_{2-s}=o(1).$$
	 	Here, $\tilde\omega_\ep$ is the translation of $ \omega_\ep$ with $\int x \tilde\omega_\ep=0$ and $R$ is the uniform constant such that $ \text{spt}(\tilde\omega_\ep)\subset B(0,R)$.
	 \end{lemma}
	 \begin{proof}
	 	Notice that $\zeta^\varepsilon \in  \mathcal{R}(\omega_\ep)$ is a rearrangement of $\omega_\ep$. So, we have $\int \zeta^\varepsilon=\int \omega_\ep=\kappa$ and $\int J(\zeta^\varepsilon)=\int J(\omega_\ep)$. Then, we deduce
	 	$$E_\ep(\zeta^\varepsilon)=\tilde E_\ep(\zeta^\varepsilon)-\int J(\zeta^\varepsilon)\geq \tilde E_\ep(\omega_\ep)-\int J(\omega_\ep)=E_\ep(\omega_\ep)\geq I_0+O(\ep^{2-2s}),$$ which implies that $\{\zeta^\ep\}$ is a maximizer sequence of $E_0$ over $\mathcal A_0$. Therefore, Theorem \uppercase\expandafter{\romannumeral2}.2 and Corollary \uppercase\expandafter{\romannumeral2}.1 in \cite{L84} provide a subsequence of $\{\zeta^\varepsilon\}$ and a sequence of points $\{z_\ep\}$ such that
	 	$$\|\zeta^\ep-\omega_0(\cdot-z_\ep)\|_1+\|\zeta^\ep-\omega_0(\cdot-z_\ep)\|_{2-s}=o(1),\ \ \text{as}\ \ep\to0.$$
	 	In view of Lemma \ref{lem2-25}, we obtain
	 	$$\|\zeta^\ep-\tilde\omega_\ep(\cdot-z_\ep)\|_1+\|\zeta^\ep-\tilde\omega_\ep(\cdot-z_\ep)\|_{2-s}=o(1),\ \ \text{as}\ \ep\to0.$$
	 	Since $\zeta^\ep$ and $\tilde\omega_\ep$ are even and non-increasing in $x_2$ and $\zeta^\ep$ is supported in $[0, M_\ep]\times [-N_\ep,N_\ep]$, we may take $z_\ep\in [-R, M_\ep+R]\times \{0\}$. This completes the proof of this lemma.
	 \end{proof}
	
	 \begin{lemma}\label{lem4-16}
	 	Let $z_\ep=(z_{\ep,1},0)$  be as in Lemma \ref{lem4-15}. Then  $z_{\ep,1}\to +\infty$ as $\ep\to0$.
	 \end{lemma}
	 \begin{proof}
	 	By the inequality $\tilde E_\ep(\zeta^\varepsilon)\geq \tilde E_\ep(\omega_\ep)$ and $M_\ep=O(\ep^{-1})$ due to the proof of Lemma \ref{lem4-12}, we get
	 	\begin{equation}\label{4-23}
	 		\int \zeta^\ep\mathcal G_s^+ \zeta^\ep\geq \int \omega_\ep\mathcal G_s  \omega_\ep+o(1).
	 	\end{equation}
	 	On the other hand, by Lemma \ref{lem4-15}, we have
	 	\begin{equation*}
	 		\begin{split}
	 			\int \zeta^\ep\mathcal G_s^+ \zeta^\ep&=\int\int \tilde\omega_\ep(x-z_\ep)\left(\frac{c_s}{|x-y|^{2-2s}}-\frac{c_s}{|x-\bar y|^{2-2s}}\right)\tilde\omega_\ep(y-z_\ep) dxdy+o(1)\\
	 			&=\int \omega_\ep\mathcal G_s  \omega_\ep- \int\int  \frac{c_s\tilde\omega_\ep(x ) \tilde\omega_\ep(y )}{|x-\bar y+2z_\ep|^{2-2s}}  dxdy+o(1),
	 		\end{split}
	 	\end{equation*}
	 	which, combined with \eqref{4-23}, implies
	 	$$\int\int  \frac{c_s\tilde\omega_\ep(x ) \tilde\omega_\ep(y )}{|x-\bar y+2z_\ep|^{2-2s}}  dxdy=o(1).$$
	 	Thus, we must have $z_\ep=(z_{\ep,1},0)$ with $z_{\ep,1}\to +\infty$ as $\ep\to0$.
	 \end{proof}
	
	 Denote $\psi_\ep(x)=\mathcal G_s \zeta^\ep$ and $\psi_0=\mathcal G_s \omega_0$. Since $\|\zeta^\ep-\omega_0(\cdot-z_\ep)\|_1+\|\zeta^\ep-\omega_0(\cdot-z_\ep)\|_{2-s}=o(1)$ and $\|\zeta^\ep\|_\infty\leq \|\omega_\ep\|_\infty\leq C$, using Lemma \ref{lem2-1} and a  similar bootstrap argument as the proof of Corollary \ref{lem2-6}, we derive
	 $$\|\psi_\ep- \psi_0(\cdot-z_\ep)\|_\infty=o(1).$$
	
	 To continue our discussion, we next find a lower bound for the Lagrange multipliers.
	
	 \begin{lemma}\label{lem4-17} One has for some constant $\tilde\mu>0$,
	 	$$\liminf_{\ep\to0} \tilde \mu_\ep\geq \tilde\mu>0.$$ 
	 \end{lemma}
	 \begin{proof}
	 	Let $R$ be the constant such that $\text{spt}(\tilde \omega_\ep)\subset B(0,R)$ for every $\ep$ small. Since $\zeta^\ep$ is a rearrangement of $\omega_\ep$, there is a point $\tilde x\in B(z_\ep,2R)\setminus\text{spt}(\zeta^\ep)$. At the point $\tilde x$, by Lemmas \ref{lem4-14} and \ref{lem4-15}, we have
	 	\begin{equation*}
	 		\begin{split}
	 			\tilde \mu_\ep&\geq \mathcal{G}_s \zeta^\ep(\tilde x)-\int \frac{c_s \zeta^\ep(y)}{|\tilde x-\bar y|^{2-2s}} dy - W\ep^{3-2s}\tilde x_1\\
	 			&\geq \psi_0(\tilde x-z_\ep)-\int \frac{c_s \omega_0(y)}{|\tilde x-\bar y+z_\ep|^{2-2s}} dy+o(1)\\
	 			&\geq \inf_{x\in B(0, 2R)}\psi_0(x)+o(1),
	 		\end{split}
	 	\end{equation*}
	 	which proves this lemma by taking $\tilde \mu =\frac{1}{2}\inf_{x\in B(0, 2R)}\psi_0(x).$
	 \end{proof}

	 \begin{corollary}\label{lem4-18}
	 	There is a constant $\tilde R>0$ such that $\text{spt}(\zeta^\ep)$ is contained in a disk of radius $\tilde R$ for arbitrary $\ep$ small.
	 \end{corollary}
	 \begin{proof}
	 	Using Lemmas  \ref{lem4-14} and   \ref{lem4-17}, one can prove this lemma through an argument similar to Lemma \ref{lem2-22}, so we omit the details.
	 \end{proof}
	
	 Let $\tilde x_\ep:=\kappa^{-1}\int x\zeta^\ep$ be the center of mass of $\zeta^\ep$. We have the following estimate of the location.
	 \begin{corollary}\label{lem4-19} It holds
	 	$$|\ep \tilde x_\ep-( d_0,0)|=o(1).$$
	 \end{corollary}
	 \begin{proof}
	 	The proof is similar to Lemma \ref{lem2-23}, so we omit the details.
	 \end{proof}

     Now, we are able to prove Theorem \ref{thm4-10}.

     \noindent{\bf Proof of Theorem \ref{thm4-10}:}
     The claims in Theorem \ref{thm4-10} follow from the above lemmas. \qed

     Using Theorem \ref{thm4-10} and the uniqueness result Theorem \ref{thmU}, we obtain the following conclusion about the set of maximizers.
	 \begin{proposition}\label{lem4-20}
	 	For $\ep>0$ sufficiently small, one has $$\tilde \Sigma_\ep=\{\omega_\ep(\cdot+c\mathbf{e}_2)\mid c\in\mathbb{R}\}.$$
	 \end{proposition}
     \begin{proof}
     	Recall that  $\Sigma_{\varepsilon}\subset 	\mathcal{A}_\ep$ denotes the set of maximizers of $E_\ep$ defined by \eqref{energy} over $\mathcal A_\ep$. The uniqueness result Theorem \ref{thmU} states that
     	$\Sigma_{\varepsilon}=\left\{ \omega_\ep(\cdot+c\mathbf{e}_2)\mid c\in\mathbb{R}\right\}.$
     	
     	  Then, by Theorem \ref{thm4-10}, we have $\emptyset\neq \tilde \Sigma_\ep\subset \mathcal{R}(\omega_\ep)$ and for arbitrary $\zeta^\ep\in\tilde \Sigma_\ep$, it holds $\text{spt}(\zeta^\ep)\subset B(\ep^{-1}(d_0,0), \ep^{-1} d_0/2).$
     	This implies $\tilde\Sigma_\ep\subset \mathcal{A}_\ep$.
     	Notice that $ \int  J( \zeta)dx=\int  J( \omega_\ep)dx, \quad \forall \zeta\in \mathcal{R}(\omega_\ep)$ due to the property of rearrangement. Then one can see easily that
     	$$\tilde\Sigma_\ep= \Sigma_{\varepsilon}=\left\{ \omega_\ep(\cdot+c\mathbf{e}_2)\mid c\in\mathbb{R}\right\}.$$
     \end{proof}

Having made all the preparation, we are now ready to give proof of  Theorem \ref{thmS} and end our paper.

    \noindent{\bf Proof of Theorem \ref{thmS}:} The orbital stability of $\omega_\ep$ follows from a combination of Theorem \ref{Sset} and Proposition \ref{lem4-20}.   \qed \\
    \vspace{0.2cm} 
\noindent{\bf Acknowledgments:} This work was supported by NNSF of China (Grant 11831009).

     \phantom{s}
     \thispagestyle{empty}

\end{document}